\documentclass[12pt]{amsart}%

\usepackage{amssymb}
\usepackage{amsmath}
\usepackage{amsthm}
\usepackage{amscd}
\usepackage[margin=1in]{geometry}
\usepackage{hyperref}

\theoremstyle{plain}
\newtheorem{theorem}{Theorem}[section]

\newtheorem{corollary}{Corollary}[section]
\newtheorem{lemma}{Lemma}[section]

\theoremstyle{remark}
\newtheorem{definition}{Definition}[section]
\newtheorem{remark}{Remark}[section]
\newtheorem{examples}{Examples}[section]

\DeclareMathOperator{\dist}{dist}
\DeclareMathOperator{\supp}{supp}

\DeclareMathOperator{\diag}{diag}

\DeclareMathOperator{\dom}{dom}
\DeclareMathOperator{\diverg}{div}
\DeclareMathOperator{\curl}{curl}

\begin{document}

\title{Magnetic energies and Feynman-Kac-It\^o formulas for symmetric Markov processes}
\author{Michael Hinz$^1$}
\address{$^1$ Fakult\"at f\"ur Mathematik, Universit\"at Bielefeld, Postfach 100131, 33501 Bielefeld, Germany}
\email{mhinz@math.uni-bielefeld.de}
\thanks{$^1$Research supported in part by SFB 701 of the German Research Council (DFG)}

\date{\today}

\begin{abstract}
Given a (conservative) symmetric Markov process on a metric space we consider related bilinear forms that generalize the energy form for a particle in an electromagnetic field. 
We obtain one bilinear form by semigroup approximation and another, closed one, by using a Feynman-Kac-It\^o formula.
If the given process is Feller, its energy measures have densities and its jump measure has a kernel, then the two forms agree on a core and the second is a closed extension of the first. In this case we provide the explicit form of the associated Hamiltonian. 
%\tableofcontents
\end{abstract}
\maketitle

\section{Introduction}\label{S:Intro} 

The energy form $\mathcal{E}^{a,v}$ and Hamiltonian $H^{a,v}$ for a charged particle in $\mathbb{R}^3$ subject to a stationary magnetic field $B=-\curl a$ and a stationary electric field $E=-\nabla v$ are given by  
\begin{equation}\label{E:classmagform}
\mathcal{E}^{a,v}(f)=\frac12\int_{\mathbb{R}^3} |(\nabla+ia)f|^2dx+\int_{\mathbb{R}^3}|f|^2vdx,
\end{equation}
and
\begin{equation}\label{E:classmagHam}
H^{a,v}f=-\frac12(\nabla+ia)^2f+vf,
\end{equation}
respectively. The time evolution of the system obeys the Schr\"odinger equation for (\ref{E:classmagHam}) and is given by the group of operators $(e^{itH^{a,v}})_{t\in\mathbb{R}}$. It is also useful to study the $L^2(\mathbb{R}^3)$-semigroup $(P_t^{a,v})_{t>0}$ defined by $P_t^{a,v}=e^{tH^{a,v}}$, and an explicit representation of the latter provides information about the domain of $H^{a,v}$, \cite{Simon79}, its spectrum, see e.g. \cite{G10, Shi87, U94}, long-time behaviour and heat kernels, \cite{E95, G10, IShiT85, LT97, Ma82}, diamagnetic inequalities \cite{GKS, LT97, Si76, Simon79}, inequalities of Kato, Golden-Thompson or Lieb-Thirring type, \cite{E94, E95, G10, GKS}.   
For $a\equiv 0$ the operators $P_t^{0,v}$ are expressed in terms of Brownian motion $B=(B_t)_{t\geq 0}$ on $\mathbb{R}^3$ using the \emph{Feynman-Kac formula}, $P_t^{0,v}f(x)=\mathbb{E}_x\left[e^{-\int_0^t v(B_s)ds}f(B_t)\right]$,
see e.g. \cite{KarShr91, Simon79}. Nonzero $a$ leads to a similar formula,
\begin{equation}\label{E:FKI0}
P_t^{a,v}f(x)=\mathbb{E}_x\left[e^{i\int_0^t a(B_s)\circ dB_s-\int_0^t v(B_s)ds}f(B_t)\right],
\end{equation}
referred to as the \emph{Feynman-Kac-It\^o formula}. Here $\int_0^t a(B_s)\circ dB_s$ is the \emph{Stratonovich integral} of $(a(B_t))_{t\geq 0}$ with respect to $B$. In the present article we provide versions of (\ref{E:classmagform}), (\ref{E:classmagHam}) and (\ref{E:FKI0}) for symmetric Markov processes on metric spaces.

Various generalizations of  (\ref{E:classmagform}) and (\ref{E:classmagHam}) are known. Classical results for magnetic Hamiltonians on Euclidean spaces can be found in \cite{BroHL00, RS, Simon79}, for cases with more singular potentials see e.g. \cite{RS, Si73}. Magnetic Hamiltonians on manifolds are studied in \cite{BrMiSh02, BGPa, GP13, GK12, G10, G10Diss, G12, Shi87, Sh01}, on lattices and graphs in \cite{AJ09, B92, CTT10b, DM06, GMT14, HS99, LLPP15, Mil11, Sh94, Su94}, and on quantum graphs in \cite{BGPb, Har08, KS03}. Some first results for fractals may be found in \cite{HRo15, HTb, HTc}. Formula (\ref{E:FKI0}) for divergence free fields $a$ was stated in Simon's book on functional integration, \cite[Theorems 15.5 and 21.1 and Example 16.3]{Simon79}. During the eighties it was discussed by Gaveau, Ikeda, Malliavin and others, \cite{GV81, GM83, IShiT85, Ma82}. On Euclidean domains it was extended to more general vector fields $a$ by Broderix, Hundertmark und Leschke, \cite[Proposition 2.9]{BroHL00}. Shigekawa \cite{Shi87} established it on compact Riemannian manifolds, G\"uneysu \cite{G10Diss, G10, G12} proved it for stochastically and geodesically complete  and  incomplete cases. In \cite{GKS} G\"uneysu, Keller and Schmidt proved (\ref{E:FKI0}) for magnetic operators on graphs. See also \cite{GMT14, GP13} and \cite{AMU98, AU00, M82} for related results.

We consider generalizations of (\ref{E:classmagform}), (\ref{E:classmagHam}) and (\ref{E:FKI0}) where a $\mu$-symmetric Hunt process $Y=(Y_t)_{t\geq 0}$ on a locally compact separable metric space $(X,\varrho)$, endowed with a Radon measure $\mu$ with full support, replaces the Brownian motion. The process $Y$ may have jumps. We make use of the regular symmetric Dirichlet form $(\mathcal{E},\mathcal{F})$ uniquely associated with $Y$, \cite{ChFu12, FOT94}, and to improve the visibility of some conceptual ideas, assume it is \emph{conservative}. The energy form (\ref{E:classmagform}) and the Hamiltonian (\ref{E:classmagHam}) can be generalized using the \emph{first order theory for Dirichlet forms} as introduced by Sauvageot in \cite{S90} and Cipriani and Sauvageot in \cite{CS03}. It provides an abstract \emph{first order derivation $\partial$} taking functions into members of a certain Hilbert space $(\mathcal{H}, \left\langle\cdot,\cdot\right\rangle_{\mathcal{H}})$. Its probabilistic counterpart is the \emph{stochastic calculus for additive functionals}, \cite{ChFiKuZh08a, FOT94, N85}, which allows to generalize (\ref{E:FKI0}). 

The \emph{Beurling-Deny decomposition} (see \cite[Section 3.2]{FOT94}) of $(\mathcal{E},\mathcal{F})$ into a strongly local part and a pure jump part carries over to the space $\mathcal{H}$, \cite{CS03}. We write $\omega=\omega_c+\omega_j$ for the decomposition of an element $\omega\in\mathcal{H}$ into a local part $\omega_c$ and a jump part $\omega_j$. The jump part $\omega_j$ may be viewed as a function on $X\times X\setminus \diag$, the product space minus the diagonal. By $\mathcal{H}_a$ we denote the closed subspace of $\mathcal{H}$ consisting of elements $a$ with antisymmetric jump part, $a_j(x,y)=-a_j(y,x)$. To $\mathcal{H}_a$ we refer as \emph{space of generalized $L^2$-differential $1$-forms} (or $L^2$-vector fields). For certain elements $a$ of $\mathcal{H}_a$ which may be viewed as real valued functions on $X\times X\setminus\diag$ and for real valued Borel functions $v$ we can introduce the \emph{magnetic energy form}
\[\mathcal{E}^{a,v}(f):=\lim_{t\to 0}\frac{1}{2t}\int_X\int_X|f(x)-e^{ia(x,y)}f(y)|^2P_t(x,dy)\mu(dx)+\int_X |f|^2vd\mu.\]
For general real elements $a\in\mathcal{H}_a$ magnetic energy forms $f\mapsto \mathcal{E}^{a,v}(f)$ can be defined via approximation. We obtain a corresponding Beurling-Deny decomposition,
\begin{equation}\label{E:firstdecomp}
\mathcal{E}^{a,v}(f)=\mathcal{E}_c^a(f)+\int_{X\times X\setminus\diag} |f(x)-e^{ia_j(x,y)}f(y)|^2J(d(x,y))+\int_X|f|^2vd\mu,
\end{equation}
see Lemma \ref{L:magneticBD} below. Here the strongly local part $\mathcal{E}_c^a$ is of the form already encountered in \cite[Section 4]{HTb} and, generalizing the first summand in (\ref{E:classmagform}), can be expressed using the local part $\partial_c$ of the first order derivation $\partial$ in the sense of \cite{CS03}, $\mathcal{E}^a_c(f)=\left\langle (\partial_c+ia_c)f, (\partial_c+ia_c)f\right\rangle_{\mathcal{H}}$.

\emph{Fukushima's decomposition theorem}, \cite[Theorem 5.2.2]{FOT94}, 
states that for a continuous function $f\in\mathcal{F}$ we have $f(Y_t)-f(Y_0)=M_t^f+N_t^f$ with  a martingale additive functional $M=(M_t)_{t\geq 0}$ of finite energy and a continuous additive functional $N=(N_t)_{t\geq 0}$ of zero energy. The space $\mathring{\mathcal{M}}$ of martingale additive functionals of finite energy is a Hilbert space, \cite[Section 5]{FOT94}, and \emph{Nakao's theorem} (Theorem \ref{T:Nakao}) provides an isometric isomorphism $\Theta$ from $\mathcal{H}$ onto $\mathring{\mathcal{M}}$ such that the image $\Theta(g\partial f)$ of an element $g\partial f$ of $\mathcal{H}$ is a stochastic integral $g\bullet M^f$ of It\^o type. For diffusions on manifolds this was proved in \cite[Theorem 5.1]{N85}, and in the same paper a divergence functional $\Lambda$ on $\mathring{\mathcal{M}}$ with values in a certain  space $\mathcal{N}^\ast_c$ of continuous additive functionals was introduced, see Section \ref{S:AF} or \cite{ChFiKuZh08a, N85}. We define the \emph{Stratonovich line integral of $\omega\in\mathcal{H}$} along the path of $Y$ by
\begin{equation}\label{E:Stratofirst}
\int_{Y([0,t])}\omega:=\Theta(\omega)_t+\Lambda(\Theta(\omega))_t,
\end{equation}
see Definition \ref{D:Stratonovich}. For diffusions on manifolds (\ref{E:Stratofirst}) was stated as a theorem in \cite[Theorem 5.2]{N85}, based on a different definition in terms of local coordinates. Definition (\ref{E:Stratofirst}) is related to Stratonovich integrals considered by Kuwae in \cite{Ku10}, however, line integrals of $1$-forms are not discussed there.

Given a real element $a\in\mathcal{H}_a$ and a suitable function $v$ we can then define 
\begin{equation}\label{E:FKIfirst}
P_t^{a,v}f(x):=\mathbb{E}_x[e^{i\int_{Y([0,t])}a-\int_0^tv(Y_s)ds}f(Y_t)]
\end{equation}
for bounded Borel functions $f$. As $a_j$ is antisymmetric (\ref{E:FKIfirst}) yields a strongly continuous and semigroup $(P_t^{a,v})_{t>0}$ of bounded self-adjoint operators on $L^2(X,\mu)$, see Theorem \ref{T:semigroup}. Consequently there exists an associated closed quadratic form $(\mathcal{Q}^{a,v}, \dom \mathcal{Q}^{a,v})$ on $L^2(X,\mu)$. It provides another generalization of (\ref{E:classmagform}), and its Hamiltonian $(H^{a,v}, \dom H^{a,v})$ generalizes (\ref{E:classmagHam}). A priori the forms $\mathcal{E}^{a,v}$ and $\mathcal{Q}^{a,v}$ may not be related. In Theorem \ref{T:coincide} we assume that $(\mathcal{E},\mathcal{F})$ admits a carr\'e du champ in the sense of \cite[Chapter I]{BH91}, i.e. that all energy measures $\Gamma(f)$, \cite{FOT94, LeJan78, Si74}, are absolutely continuous with respect to $\mu$. We assume that the jump measure $J$ of $(\mathcal{E},\mathcal{F})$ is of form $J(d(x,y))=n(x,dy)\mu(dy)$ with kernel $n(x,dy)$, an assumption void in the strongly local case. Whenever $f$ satisfies certain smoothness and integrability conditions and $g\in \mathcal{F}\cap L^\infty(X,\mu)$, the dual $L^1$-$L^\infty$ pairing $\left\langle H^{a,v}f,g\right\rangle$ then equals
\begin{equation}
\left\langle (\partial_c+ia_c)f,(\partial_c+ia_c)g\right\rangle_{\mathcal{H}}
+\int_X\int_X(f(x)-e^{ia_j(x,y)}f(y))n(x,dy)\overline{g(x)}\mu(dx)+\left\langle vf,g\right\rangle.
\end{equation}
In Theorem \ref{T:coincide2} we assume in addition that the semigroup $(P_t)_{t>0}$ associated with $(\mathcal{E},\mathcal{F})$ is Feller and that there is an $L^2(X,\mu)$-dense subspace $\mathcal{C}_L$ of the domain of the Feller generator consisting of compactly supported functions. Under these assumptions the forms $\mathcal{Q}^{a,v}$ and $\mathcal{E}^{a,v}$ agree on $\mathcal{C}_L$ and $(\mathcal{Q}^{a,v}, \dom \mathcal{Q}^{a,v})$ is a closed extension of 
$(\mathcal{E}^{a,v}, \mathcal{C}_L)$. 

In \cite{HRo15, HTb} we used a very general argument (the KLMN theorem, \cite[Theorem X.17]{RS}) to conclude the closability of magnetic energy forms, seen as a perturbation of a given (closed) Dirichlet form, respectively. This method implies that the domain of the magnetic form equals the domain of the original form with zero magnetic field, but it needs some additional $L^\infty$-boundedness assumptions on the potential $a$. In the present situation the domain of $\mathcal{Q}^{a,v}$ (even for $v\equiv 0$) may be quite different from the domain of the form $\mathcal{Q}^{0,0}=\mathcal{E}$ with zero magnetic and electric potentials. See also \cite[Section X.3, Example 4]{RS} for another situation where a nonzero magnetic perturbation may change the energy domain. In contrast to our results in \cite{HRo15, HTb} we do not assume $a$ to satisfy any sort of $L^\infty$-bound.

For the Hamiltonian $H^{a,v}$ associated with $\mathcal{Q}^{a,v}$ Theorem \ref{T:coincide2} yields the representation
\[H^{a,v}f(x)=(\partial_c+ia_c)^\ast (\partial_c+ia_c)f(x)+\int_X(f(x)-e^{ia_j(x,y)}f(y))n(x,dy)+v(x)f(x),\]
$f\in\mathcal{C}_L$, generalizing (\ref{E:classmagHam}). We observe a semigroup approximation for $H^{a,v}$, Corollary \ref{C:sgapproxmagnetic} and a diamagnetic inequality, Corollary \ref{C:diamagenergy}. 

Theorems \ref{T:coincide} and \ref{T:coincide2} follow by a typical Girsanov-type argument, \cite{Simon79}. The effort done here consists in providing the necessary setup by systematically combining the first order theory from \cite{CS03} and the stochastic analysis of martingale additive functionals. To prove the symmetry of the operators $P_t^{a,v}$ in  a general situation we verify the time-antisymmetry of (\ref{E:Stratofirst}). It was known before for the diffusion case, \cite{Fi93}, \cite{Simon79}, and our result, Theorem \ref{T:antisymm}, now allows $Y$ to have jumps. If $M=\Theta(\omega)$ is the martingale additive functional associated with  $\omega\in\mathcal{H}$ under Nakao's isomorphism $\Theta$, then  Theorem \ref{T:jumpfunction} states that $M_t-M_{t-}=\omega_j(Y_{t-},Y_t)$ for every $t>0$ $\mathbb{P}_\mu$-a.s. That is, the \emph{jump function of $M$}, \cite{ChFiKuZh08a, Sh88}, is given by the jump part $\omega_j$ of $\omega$. If $\omega_j$ is antisymmetric we can use a representation of the divergence functional $\Lambda$ shown in \cite[Definition 3.3 and Theorem 3.6]{ChFiKuZh08a} to obtain the antisymmetry of the integral. For background and related results see \cite{ChFiKuZh08a, ChFiKuZh08b, ChFiTaYiZh, ChZh02, Fi93, FiKu04, FOT94, LLZh98, LZh94, N85}.

Our results apply to diffusions and jump processes on Euclidean spaces, domains, manifolds, graphs and quantum graphs and also to processes on metric spaces \cite{Ba, ChKu06, Ki01}, for which the related vector analysis had been investigated in \cite{CS09, HKT, HRT, HTa, HTb, HTc, IRT}.

In Sections \ref{S:Alg} and \ref{S:jumps} we briefly survey an algebraic point of view and connect it with jump measures to illustrate some features of the first order calculus. In Section \ref{S:Dforms} we consider conservative regular Dirichlet forms and take a semigroup perspective upon $1$-forms to introduce energy forms $\mathcal{E}^{a,v}$ with magnetic and electric potential. We discuss Nakao's theorem in Section \ref{S:AF}, Nakao's divergence and Stratonovich line integrals in Section \ref{S:lineintegral} and time-antisymmetry in Section \ref{S:timereversal}. In Section \ref{S:FKI} the Feynman-Kac-It\^o formula (\ref{E:FKIfirst}) is used to define a strongly continuous self-adjoint semigroup. Section \ref{S:Ident} shows that under the mentioned hypotheses this semigroup is associated to an extension of the form $\mathcal{E}^{a,v}$. Given a symmetric bilinear expression such as $\mathcal{E}(f,g)$ we agree to write $\mathcal{E}(f):=\mathcal{E}(f,f)$.

\subsection*{Acknowledgment}
The author thanks Shiping Liu and Alexander Teplyaev for stimulating discussions and the anonymous referee for careful reading and immensely helpful criticism.

\section{Algebraic preliminaries}\label{S:Alg}

Let $X$ be a nonempty set and let $\mathcal{D}$ be an algebra of bounded complex valued functions on $X$ (with pointwise multiplication). First assume that $\mathcal{D}$ is unital, i.e. that $\mathbf{1}\in\mathcal{D}$ (otherwise we can use the unitisation of $\mathcal{D}$, see below).  Consider the \emph{universal derivation} $d:\mathcal{D}\to\mathcal{D}\otimes\mathcal{D}$, defined by 
\begin{equation}\label{E:abstractderivation}
df:=f\otimes \mathbf{1} - \mathbf{1}\otimes f, \ \ f\in\mathcal{D}.
\end{equation}
In particular, $d\mathbf{1}=0$. With left and right actions of $\mathcal{D}$ on $\mathcal{D}\otimes\mathcal{D}$ given by
\begin{equation}\label{E:abstractactions}
h(f\otimes g):=(fh)\otimes g\ \ \text{and}\ \ (f\otimes g)h:=f\otimes (gh)
\end{equation}
we observe $d(fg)=fdg+(df)g$, $f,g\in\mathcal{D}$. By $\Omega^1(\mathcal{D})$ we denote the subbimodule of $\mathcal{D}\otimes\mathcal{D}$ generated by the elements $fdg$ and $(df)g$. It agrees with the kernel of the multiplication operator from $\mathcal{D}\otimes\mathcal{D}$ onto $\mathcal{D}$, obtained as the linear extension of $f\otimes g\mapsto fg$. With the interpretation $(\sum_i f_i\otimes g_i)(x,y)=\sum_i f_i(x)g_i(y)$ as bounded complex valued functions on $X\times X$ we may view the elements of $\Omega^1(\mathcal{D})$ as functions on $X\times X$ that vanish on the diagonal and in particular, 
\begin{align}\label{E:elementary}
df(x,y)&=f(x)-f(y)\\
gdf(x,y)&=g(x)(f(x)-f(y))\notag\\
(df)g(x,y)&=g(y)(f(x)-f(y))\notag .
\end{align}
In the next section we will consider the elements of $\Omega^1(\mathcal{D})$ as functions on $X\times X\setminus \diag$, where $\diag:=\left\lbrace (x,x): x\in X\right\rbrace$ denote the diagonal in $X\times X$.

The antisymmetrizer $\mathcal{A}$ is defined as the linear operator $\mathcal{A}:\Omega^1(\mathcal{D})\to\Omega^1(\mathcal{D})$ given by
\begin{equation}\label{E:anti}
\mathcal{A}\omega(x,y):=\frac12(\omega(x,y)-\omega(y,x)), 
\end{equation}
$\omega\in \Omega^1(\mathcal{D})$. We consider its image $\Omega^1_a(\mathcal{D}):=\mathcal{A}(\Omega^1(\mathcal{D}))$ and write again $\omega$ to denote $\mathcal{A}\omega$, seen as an element of $\Omega^1_a(\mathcal{D})$. Restricted to this space, the actions (\ref{E:abstractactions}) agree, i.e. for any $\omega\in \Omega_a^1(\mathcal{D})$ and any $h\in\mathcal{D}$ we have 
\[(h\omega)(x,y)=\frac12(h(x)+h(y))\omega(x,y)=(\omega h)(x,y),\] 
seen as equalities in $\Omega^1_a(\mathcal{D})$, and in particular, $gdf=(df)g$. To $\Omega^1_a(\mathcal{D})$ one usually refers as \emph{space of universal $1$-forms}. See for instance \cite[Section 8.1]{GVF}.

If $\mathcal{D}$ does not contain $\mathbf{1}$ we consider the unitisation $\mathcal{D}_e$ of $\mathcal{D}$, given by $\mathcal{D}_e:=\mathcal{D}\oplus\mathbb{C}$ with multiplication $(f,\lambda)(g,\nu):=(fg+\lambda g+\nu f,\lambda\nu)$ for all $f,g\in\mathcal{D}$, $\lambda,\nu\in\mathbb{C}$. Its unit element is $(0,1)$. Viewing $\mathbb{C}\times\mathbb{R}$ with the product $(a,b)(c,d):=(ac+ad+bc,bd)$ we can consider the multiplication in $\mathcal{D}_e$ as pointwise operation. The map $f\mapsto (f,0)$ provides an injection of $\mathcal{D}$ into $\mathcal{D}_e$. We therefore identify $f$ and $(f,0)$, in this sense $\mathcal{D}$ may be seen as an ideal of $\mathcal{D}_e$. As before we can define actions of $\mathcal{D}_e$ on $\mathcal{D}_e\otimes\mathcal{D}_e$, introduce the universal derivation by $d(f,\lambda)=(f,\lambda)\otimes (0,1)-(0,1)\otimes (f,\lambda)$, and consider the subbimodule $\Omega^1(\mathcal{D}_e)$. Again the elements $(g,0)d(f,0)$ of $\Omega^1(\mathcal{D}_e)$ are viewed as functions on $X\times X$, for instance $((g,0)d(f,0))(x,y)=(g(x)(f(x)-f(y),0)$ will be identified with $(x,y)\mapsto g(x)(f(x)-f(y))$. In this sense we may assume that $df$, $gdf$ and $(df)g$ can be written as in (\ref{E:elementary}).

\section{Non-local forms and magnetic potentials}\label{S:jumps}

In this section let $X$ be a locally compact separable Hausdorff space and  $\mathcal{D}$ an algebra of bounded functions on $X$ such that $C_c(X)\cap \mathcal{D}$ is uniformly dense in $C_c(X)$. Suppose that $J$ is a symmetric nonnegative Radon measure on $X\times X\setminus\diag$ such that for all $f\in\mathcal{D}$ the elements $df$ are square integrable with respect to $J$. Then all elements $\omega$ of $\Omega^1(\mathcal{D})$ are $J$-square integrable functions $(x,y)\mapsto \omega(x,y)$ on $X\times X\setminus \diag$, and 
\begin{equation}\label{E:L2Jseminorm}
\left\|\omega\right\|_{L^2(X\times X\setminus\diag,J)}^2=\int_{X\times X\setminus\diag}|\omega(x,y)|^2J(d(x,y))
\end{equation}
defines a Hilbert seminorm on $\Omega^1(\mathcal{D})$. Whether the seminorm of an element (\ref{E:elementary}) of $\Omega^1(\mathcal{D})$ is zero or not depends on $f$, $g$ and the structure of $J$. 

\begin{examples}
Consider $X=\mathbb{R}^n$ and $\mathcal{D}=C_c^1(\mathbb{R}^n)$, the space of compactly supported continuously differentiable functions on $\mathbb{R}^n$. Given $0<\alpha<2$ and $\varepsilon>0$, let $J$ be given by $J(d(x,y))= \mathbf{1}_{\left\lbrace |x-y|<\varepsilon\right\rbrace}\:|x-y|^{-n-\alpha}\:dxdy$. If $f$ is constant on $\left\lbrace x\in \mathbb{R}^n: \dist(x,\supp g)<\varepsilon \right\rbrace$ then $(df)g$ has zero seminorm.
\end{examples}

\begin{lemma}
The space $\Omega^1(\mathcal{D})$ is dense in $L^2(X\times X\setminus \diag, J)$.
\end{lemma}
\begin{proof}
The closure of $\Omega^1(\mathcal{D})$ is a closed subspace of $L^2(X\times X\setminus \diag, J)$. If $\omega$ is an element of its orthogonal complement, then $\left\langle (df)g,\omega\right\rangle_{L^2(X\times X\setminus \diag, J)}=0$ for all $f,g\in C_c(X)\cap \mathcal{D}$ and $0=\int_{X\times X\setminus \diag}f(x)g(y)\overline{\omega(x,y)}J(d(x,y))$ whenever $f$ and $g$ have disjoint supports. By Stone-Weierstrass we can approximate any function from $C_c(X\times X\setminus \diag)$ uniformly by linear combinations of such functions $(x,y)\mapsto f(x)g(y)$ and, using Cauchy-Schwarz and a simple cut-off argument, also in $L^2(X\times X\setminus \diag, J)$. Hence the integral of any function from $C_c(X\times X\setminus \diag)$ with respect to the measure $\overline{\omega}\circ J$ is zero, so that $\omega=0$ in $L^2(X\times X\setminus \diag, J)$.
\end{proof} 
The left and right actions (\ref{E:abstractactions}) of $\mathcal{D}$ on $\Omega^1(\mathcal{D})$ induce left and right actions on $\mathcal{D}$ on $L^2(X\times X\setminus\diag, J)$ in a straightforward manner.

There is a more abstract way of defining the same space. By
\begin{equation}\label{E:tensorseminormJ}
\left\|f\otimes g\right\|_{\mathcal{H}}:=\left\|(df)g\right\|_{L^2(X\times X\setminus\diag, J)}
\end{equation}
we define a Hilbert seminorm on $\mathcal{D}\otimes\mathcal{D}$. Factoring out zero seminorm elements and completing again yields a Hilbert space $\mathcal{H}$, a special case of a construction proposed by Sauvageot \cite{S90} and by Cipriani and Sauvageot in \cite{CS03}.  See also \cite{CS03, CS09, HRT, HTa, HTb, HTc, IRT}. The map $\iota: \mathcal{H}\to L^2(X\times X\setminus\diag, J)$, defined as the linear and continuous extension of $\iota(f\otimes g):=(df)g$, provides an \emph{isometric isomorphism} between the two spaces. We therefore identify the elements of $\mathcal{H}$ with functions in $L^2(X\times X\setminus\diag, J)$. The preimage in $\mathcal{H}$ of $df\in \Omega^1(\mathcal{D})$ under $\iota$ is denoted by 
\begin{equation}\label{E:derivationjump}
\partial f:=f\otimes \mathbf{1}.
\end{equation}
Extending this definition we can obtain a linear map $\partial:\mathcal{D}\to \mathcal{H}$. Definitions (\ref{E:abstractactions}) motivate to declare a right action of $\mathcal{D}$ on $\mathcal{H}$ by
\begin{equation}\label{E:rightaction}
(f\otimes g)h:=f\otimes (gh)
\end{equation}
and continuous linear extension, note that for any finite linear combination $\omega=\sum_i f_i\otimes g_i$ we have 
\begin{equation}\label{E:rightactionbound}
\left\|\omega h\right\|_{\mathcal{H}}\leq \left\|h\right\|_{\sup}\left\|\omega\right\|_{\mathcal{H}},
\end{equation}
and therefore also for arbitrary $\omega\in\mathcal{H}$. Then $(x,y)\mapsto h(y)\omega(x,y)$ represents $\omega h$ and in particular, $(df)g$ represents $(\partial f)g$. The $\mathcal{H}$-class of $gdf$ is $(fg)\otimes \mathbf{1}-g\otimes f$, what suggests to modify (\ref{E:abstractactions}) and to define a left action of $\mathcal{D}$ on $\mathcal{H}$ by 
\begin{equation}\label{E:leftaction}
h(f\otimes g):=(fh)\otimes g-h\otimes (fg)
\end{equation}
and continuous linear extension. Similarly as before we have 
\begin{equation}\label{E:leftactionbound}
\left\|h \omega\right\|_{\mathcal{H}}\leq \left\|h\right\|_{\sup}\left\|\omega\right\|_{\mathcal{H}},
\end{equation}
a priori for finite linear combinations as above and consequently for all $\omega\in\mathcal{H}$. Then $(x,y)\mapsto h(x)\omega(x,y)$ represents  $h\omega$, and in particular, $gdf$ represents $g\partial f$. These definitions now yield a Leibniz rule for $\partial$,
\begin{equation}\label{E:Leibniz}
\partial(fg)=f\partial g+ (\partial f)g, \ \ f,g\in \mathcal{D}.
\end{equation}

The antisymmetrizer $\mathcal{A}$ as in formula (\ref{E:anti}) may also be seen as an orthogonal projection in $L^2(X\times X\setminus\diag,J)$, and we denote its image by $L^2_a(X\times X\setminus\diag,J)$. The space $\Omega^1_a(\mathcal{D})$ is dense in $L^2_a(X\times X\setminus \diag, J)$. To the preimage $\mathcal{H}_a$ of $L^2_a(X\times X\setminus \diag, J)$ under $\iota$ we refer as the \emph{space of differential $1$-forms associated with $J$}. Given $a\in \mathcal{H}_a$, the elements $a h$ and $ha$ agree in $\mathcal{H}_a$, both are represented by $(x,y)\mapsto \frac12(h(x)+h(y))a(x,y)$. That is, on $\mathcal{H}_a$ the left and right actions of $\mathcal{D}$ coincide.

By
\[\mathcal{E}(f,g):=\left\langle \partial f, \partial g\right\rangle_\mathcal{H}=\int_{X\times X\setminus\diag}(f(x)-f(y))\overline{(g(x)-g(y))} J(d(x,y))\]
we can define a nonnegative definite (conjugate) symmetric bilinear form on $\mathcal{D}$ and $f\mapsto \mathcal{E}(f)^{1/2}$ provides a Hilbert seminorm on $\mathcal{D}$ with respect to which $\partial$ becomes a bounded operator, $\left\|\partial f\right\|_{\mathcal{H}}=\left\|df\right\|_{L^2(X\times X\setminus\diag,J)}=\mathcal{E}(f)^{1/2}$.

\begin{examples}\mbox{}
\begin{enumerate}
\item[(i)]
For $X=\mathbb{R}^n$ and fixed $0<\alpha<2$ let $J(d(x,y))=\frac12|x-y|^{-n-\alpha}\:dxdy$ and let $\mathcal{D}$ be the space of Lipschitz functions with compact support. Then the form
\[\mathcal{E}(f)=\frac12\int_{\mathbb{R}^n}\int_{\mathbb{R}^n}\frac{|f(x)-f(y)|^2}{|x-y|^{n+\alpha}}dxdy, \ \ f\in\mathcal{D},\]
is closable on $L^2(\mathbb{R}^n)$ and (up to a constant) its closure $(\mathcal{E}, H^\alpha(\mathbb{R}^d))$ is the bilinear form associated with the \emph{fractional Laplacian} $(-\Delta)^{\alpha/2}$, \cite{Ber96, FOT94, Sa99}.
\item[(ii)] Let $(V,b,\mu)$ be a \emph{weighted graph}, \cite{GKS}, i.e. let $V\neq \emptyset$ be a countable set, $b$ a nonnegative symmetric real valued function on $V\times V$ vanishing on the diagonal and satisfying $\sum_{q\in V} b(p,q)<+\infty$ for all $p\in V$, and 
$\mu$ a positive real valued function on $V$.  We assume that $(V,b,\mu)$ is locally finite and connected, \cite{GKS}, and endow $V$ with the discrete topology. Then $\mathcal{D}=C_c(V)\oplus \mathbb{C}$ and $J(\left\lbrace p\right\rbrace,\left\lbrace q\right\rbrace)=b(p,q)$ yield the energy form
\begin{equation}\label{E:graphform}
\mathcal{E}(f,g)=\sum_{p}\sum_q b(p,q)(f(p)-f(q))\overline{(g(p)-g(q))}.
\end{equation}
\end{enumerate}
\end{examples}

Given a real valued element $a$ of $\mathcal{H}_a$ we can introduce a quadratic form 
\begin{equation}\label{E:jump}
\mathcal{E}^a(f):=\int_{X\times X\setminus\diag}|f(x)-e^{ia(x,y)}f(y)|^2 J(d(x,y)), \ \ f\in\mathcal{D},
\end{equation}
a priori it may be extended real valued. To $a$ we refer as \emph{magnetic (or vector) potential}. The modification of $df$ to include a phase as in (\ref{E:jump}), $d_af(x,y):=f(x)-e^{ia(x,y)}f(y)$, is a variant of the \emph{Peierls substitution}, see \cite[Section 2.2.3]{AkkMon07} and the references cited there.

\begin{lemma}\label{L:boundedjumpcase}
Let $a\in\mathcal{H}_a$ be real valued. Then $\mathcal{E}^a(f)\leq 4\mathcal{E}(f)+4\left\|f\right\|_{\sup}^2\left\|a\right\|_\mathcal{H}^2$, $f\in\mathcal{D}$, and $\mathcal{E}^a$ defines a (conjugate) symmetric bilinear form on $\mathcal{D}$ by polarization.
\end{lemma}

\begin{proof} The symmetry is inherited from the scalar product in $L^2(X\times X\setminus \diag, J)$. For the estimate note that
\[\mathcal{E}^a(f)\leq 4\int_X\int_X|f(x)-f(y)|^2J(dxdy)\notag\\
+4\int_X\int_X|f(y)|^2|1-e^{ia(x,y)}|^2J(dxdy),\]
and since $e^{ia(x,y)}-1=ia(x,y)\int_0^1e^{ita(x,y)}dt$, the second integral is bounded by $\left\|af\right\|_{\mathcal{H}}^2$.
\end{proof}

\section{Dirichlet forms, semigroups and magnetic potentials}\label{S:Dforms}

Now let $(X,\varrho)$ be a locally compact separable metric space, $\mu$ a nonnegative Radon measure on $X$ with full support and $(\mathcal{E},\mathcal{F})$ a regular Dirichlet form on  $L^2(X,\mu)$, \cite{ChFu12, FOT94}. A priori Dirichlet form theory is formulated for real valued (classes of) functions. We later use the natural complexification, and to keep notation short we do so without introducing new symbols. By $(P_t)_{t>0}$ we denote the associated $\mu$-symmetric Markovian semigroup, \cite{FOT94}, and write $(P_t(\cdot,\cdot))_{t>0}$ for the associated family of Markov transition kernels. Then for any $t>0$ and any bounded Borel function we have $P_tf(x)=\int_X f(y)P_t(x,dy)$ for $\mu$-a.e. $x\in X$, and for any $t>0$ the nonnegative Radon measure $\Pi_t(dxdy):=\frac{1}{t} P_t(x,dy)\mu(dx)$ 
is symmetric on $X\times X$. We assume that $(\mathcal{E},\mathcal{F})$ is \emph{conservative}, i.e. $P_t\mathbf{1}=\mathbf{1}$ for all $t>0$. Then 
\begin{equation}\label{E:sgapprox}
\mathcal{E}(f,g)=\lim_{t\to 0}\frac12\int_X\int_X(f(x)-f(y))(g(x)-g(y))\Pi_t(dxdy),
\end{equation}
$f,g\in\mathcal{F}$. Now set $\mathcal{C}:=\mathcal{F}\cap C_c(X)$. The \emph{Beurling-Deny decomposition} of $(\mathcal{E},\mathcal{F})$ reads
\begin{equation}\label{E:BD}
\mathcal{E}(f,g)=\mathcal{E}_c(f,g)+\mathcal{E}_j(f,g)
\end{equation}
$f,g\in\mathcal{C}$, where $\mathcal{E}_c$ is strongly local and \[\mathcal{E}_j(f,g)=\int_{X\times X\setminus\diag}(f(x)-f(y))(g(x)-g(y))J(d(x,y))\]
with a symmetric nonnegative Radon measure $J$ on $X\times X\setminus \diag$. Both $\mathcal{E}_c$ and $J$ are uniquely determined, \cite[Section 3.2]{FOT94}.  The next two results are straightforward.

\begin{lemma}\label{L:vague}
For any $\varepsilon>0$ and any $a\in C_c(X\times X\setminus\diag)$ we have
\[\lim_{t\to 0}\frac12\int\int_{\left\lbrace \varrho(x,y)>\varepsilon\right\rbrace}a(x,y)\Pi_t(dxdy)=\int_{\left\lbrace \varrho(x,y)>\varepsilon\right\rbrace}a(x,y)J(d(x,y)).\]
\end{lemma}

\begin{corollary}
For any $f,g\in\mathcal{C}$ we have
\[\mathcal{E}_j(f,g)=\lim_{\varepsilon\to 0}\lim_{t\to 0}\frac12\int\int_{\left\lbrace \varrho(x,y)>\varepsilon\right\rbrace}(f(x)-f(y))(g(x)-g(y))\Pi_t(dxdy)\]
and consequently also
\[\mathcal{E}_c(f,g)=\lim_{\varepsilon\to 0}\lim_{t\to 0}\frac12\int\int_{\left\lbrace \varrho(x,y)\leq \varepsilon\right\rbrace}(f(x)-f(y))(g(x)-g(y))\Pi_t(dxdy).\]
\end{corollary}

Let $\mathcal{D}$ be the complexification of $\mathcal{C}$ if $X$ is compact, otherwise use its unitisation. On $\mathcal{D}\otimes\mathcal{D}$ we can introduce a Hilbert seminorm by
\begin{align}\label{E:Hseminorm}
\left\|f\otimes g\right\|_{\mathcal{H}}&=\lim_{t\to 0}\frac12\left\|(df)g\right\|_{L^2(X\times X,\Pi_t)}\notag\\
&=\lim_{t\to 0}\left(\frac12\int_X\int_X|g(y)|^2|f(x)-f(y)|^2\Pi_t(dxdy)\right)^{1/2}
\end{align}
and bilinear extension (of the corresponding bilinear form). Let $\mathcal{H}$ denote the Hilbert space obtained by factoring out zero seminorm elements and completing, see \cite{CS03}. Then by construction $\Omega^1(\mathcal{D})$ is dense in $\mathcal{H}$. Again the $\mathcal{H}$-class of $df$ is $\partial f:=f\otimes \mathbf{1}$. Right and left actions of $\mathcal{D}$ on $\mathcal{H}$ can be defined as in (\ref{E:rightaction}) and (\ref{E:leftaction}), and the estimates (\ref{E:rightactionbound}) and (\ref{E:leftactionbound}) remain valid. The Leibniz rule (\ref{E:Leibniz}) holds, and we observe $\mathcal{E}(f,g)=\left\langle \partial f, \partial g\right\rangle_\mathcal{H}$, $f,g\in\mathcal{C}$. The operator $\partial$ extends to a densely defined closed unbounded operator $\partial: L_2(X,\mu)\to\mathcal{H}$ with domain $\mathcal{F}$. Let $\partial^\ast$ denote its adjoint. That is, $\omega\in\mathcal{H}$ is in $\dom \partial^\ast$ if there is some $g\in L_2(X,\mu)$ such that for all $f\in\mathcal{C}$ we have $\left\langle \partial f,\omega\right\rangle_{\mathcal{H}}=\left\langle f, g\right\rangle_{L^2(X,\mu)}$, and in this case, we set $\partial^\ast \omega:=g$. For all $\omega\in dom\:\partial^\ast$ we then have the integration by parts identity
\begin{equation}\label{E:IbP}
\left\langle \partial f,\omega\right\rangle_{\mathcal{H}}=\left\langle f, \partial^\ast \omega \right\rangle_{L^2(X,\mu)},\ \ f\in\mathcal{F}.
\end{equation}

Using (\ref{E:Hseminorm}) we can define the antisymmetrizer $\mathcal{A}$ as an orthogonal projection in $\mathcal{H}$, and we denote its image $\mathcal{A}(\mathcal{H})$ by $\mathcal{H}_a$. To $\mathcal{H}_a$ we refer as the \emph{space of differential $1$-forms associated with $(\mathcal{E},\mathcal{F})$}. Again the left and right actions of $\mathcal{D}$ on $\mathcal{H}_a$ coincide. The space $\Omega^1_a(\mathcal{D})$ is dense in $\mathcal{H}_a$.

The symmetry of the measures $\Pi_t$ implies
\[(\partial^\ast a)(x)=\lim_{t\to 0}\frac{1}{t}\int_X a(x,y)P_t(x,dy),\ \ a\in \Omega^1_a(\mathcal{D})\cap \dom \partial^\ast. \]
For the infinitesimal generator $(L,\dom L)$ of $(\mathcal{E},\mathcal{F})$ we have
\begin{equation}\label{E:genformula}
Lf(x)=\lim_{t\to 0}\frac{1}{t}\int_X (f(y)-f(x))P_t(x,dy), \ \ f\in\dom L,
\end{equation}
and $Lf=-\partial^\ast\partial f$. Further details can be found in \cite[Section 3]{HRT} (although with a different sign convention). As $\mathcal{H}_a$ is a generalization of the $L_2$-space of differential forms on Riemannian manifolds, we interpret the elements of $\mathcal{H}_a$ also as \emph{$L_2$-vector fields} (with the Riesz representation theorem  in mind), and the operators $\partial$ and $\partial^\ast$ as abstract \emph{gradient} and \emph{divergence} operators. 

The image $Im\:\partial=\left\lbrace \partial u: u\in\mathcal{F}\right\rbrace$ 
of the derivation is a closed subspace of $\mathcal{H}$, see for instance \cite[Section 4]{HKT}. Therefore $\mathcal{H}$ decomposes orthogonally into the range $Im\:\partial$ of $\partial$ and its orthogonal complement $(Im\:\partial)^\bot$ in $\mathcal{H}$, which by (\ref{E:IbP}) equals the kernel $\ker \partial^\ast$ of $\partial^\ast$,
\begin{equation}\label{E:Hodge}
\mathcal{H}=Im\:\partial\oplus \ker \partial^\ast.
\end{equation}
Any $\omega\in\mathcal{H}$ uniquely decomposes $\omega=\partial u+\eta$ where $u\in\mathcal{F}$ and $\eta\in \ker \partial^\ast$. Note also that for $f,g\in\mathcal{C}$ we have
\begin{equation}\label{E:energymeasures}
\left\|f\otimes g\right\|_{\mathcal{H}}^2=\int_X|g|^2d\Gamma(f)=\int_X |g|^2 d\Gamma_c(f)+\int_{X\times X\setminus \diag}|g(y)|^2|f(x)-f(y)|^2J(d(x,y))
\end{equation}
where $\Gamma(f)$ denotes the \emph{energy measure of $f\in\mathcal{C}$}, \cite{LeJan78, Si74, FOT94}, and $\Gamma_c(f)$ its strongly local part, see \cite[Lemma 3.5]{CKS87} and \cite[Section 3.2]{FOT94}. For any $\omega\in\mathcal{H}$ there is a nonnegative Radon measure $\Gamma_\mathcal{H}(\omega)$ satisfying
\begin{equation}\label{E:GammaH}
\int_X\varphi d\Gamma_\mathcal{H}(\omega)=\left\langle \varphi \omega, \omega\right\rangle_\mathcal{H},\ \ \varphi\in\mathcal{C},
\end{equation}
\cite[Section 2]{HRT}, note that $\Gamma_\mathcal{H}(\partial u)=\Gamma(u)$, $u\in\mathcal{C}$. Lemma \ref{L:vague} implies
$\left\|\omega\right\|_{\mathcal{H}}^2=\left\|\omega\right\|_{\mathcal{H}_c}^2+\left\|\omega\right\|_{\mathcal{H}_j}^2$, $\omega\in\Omega^1(\mathcal{D})$, where
\begin{equation}\label{E:limlocalpart}
\left\|\omega\right\|_{\mathcal{H}_c}^2=\lim_{\varepsilon\to 0}\lim_{t\to 0}\frac{1}{2}\int\int_{\left\lbrace \varrho(x,y)\leq\varepsilon\right\rbrace}|\omega(x,y)|^2\Pi_t(dxdy).
\end{equation}
and 
\[\left\|\omega\right\|_{\mathcal{H}_j}^2:=\int_{X\times X\setminus\diag} |\omega(x,y)|^2 J(d(x,y)).\]
This induces the orthogonal decomposition 
\begin{equation}\label{E:BDH}
\mathcal{H}=\mathcal{H}_c\oplus\mathcal{H}_j
\end{equation}
of $\mathcal{H}$, where $\mathcal{H}_c:=\ker \left\|\cdot\right\|_{\mathcal{H}_j}$ and $\mathcal{H}_j$ is its orthogonal complement. Each $\omega\in\mathcal{H}$ may therefore uniquely be written as $\omega=\omega_c+\omega_j$ with $\omega_c\in\mathcal{H}_c$ and $\omega_j\in\mathcal{H}_j$. The nonlocal part $\omega_j$ may be viewed as a member of $L^2(X\times X\setminus\diag, J)$. By (\ref{E:limlocalpart}) the local part $\omega_c$ vanishes outside any given neighborhood of the diagonal, and the space $\mathcal{H}_c$ is seen to be invariant under the antisymmetrizer $\mathcal{A}$. Consequently  $\omega\in\mathcal{H}$ is an element of $\mathcal{H}_a$ if and only if $\omega_j$ is antisymmetric. For $\omega=\partial u$ we use the notation $\partial u=\partial_cu+\partial_ju$. Note that 
$\mathcal{E}_c(f,g)=\left\langle \partial_c f, \partial_c g\right\rangle_\mathcal{H}$ and $\mathcal{E}_j(f,g)=\left\langle \partial_j f, \partial_j g\right\rangle_\mathcal{H}$.
The element $\partial_jf$ is represented by $df$, seen as an element of $L^2(X\times X\setminus \diag, J)$. By
\begin{equation}\label{E:GammaHc}
\int_X\varphi d\Gamma_{\mathcal{H},c}(\omega)=\left\langle \varphi \omega, \omega\right\rangle_{\mathcal{H}_c},\ \ \varphi\in\mathcal{C},
\end{equation}
we obtain the local part $\Gamma_{\mathcal{H},c}(\omega)$ of $\Gamma_{\mathcal{H}}(\omega)$ of $\omega\in\mathcal{H}$. In particular, $\Gamma_{\mathcal{H},c}(\partial u)=\Gamma_c(u)$, $u\in\mathcal{C}$.

\begin{remark} If the Dirichlet form $(\mathcal{E},\mathcal{F})$ is of pure jump type, i.e. $\mathcal{E}_c\equiv 0$, then $\mathcal{H}=\mathcal{H}_j$ and with the jump measure $J$ as in (\ref{E:BD}) we are in a situation discussed in Section \ref{S:jumps}. If on the other hand $(\mathcal{E},\mathcal{F})$ is strongly local, i.e. $\mathcal{E}_j\equiv 0$, then $\mathcal{H}=\mathcal{H}_c$.
\end{remark}

Also $\partial_c$ and $\partial_j$ are densely defined and have adjoints $\partial_c^\ast$ and $\partial_j^\ast$. In particular, 
\[\left\langle \partial^\ast_ja, f\right\rangle_{L_2(X,\mu)}=2\int_{X\times X\setminus\diag}f(x)a_j(x,y)J(d(x,y))\]
for any $a\in \mathcal{H}_a\cap\dom\partial_j^\ast$ and if $J(dxdy)=\frac12n(x,dy)\mu(dx)$ with a kernel $n(x,dy)$ on $(X,\mathcal{B}(X))$,
\[\partial_j^\ast a(x)=\int_X a_j(x,y)n(x,dy).\]
In this case $Lf=L_cf+L_jf$, $f\in \dom L$, where $L_c$ is strongly local and
\begin{equation}\label{E:Lj}
L_jf(x)=\int_X (f(y)-f(x))n(x,dy). 
\end{equation}

Together the orthogonal decompositions (\ref{E:Hodge}) and (\ref{E:BDH}) yield
\begin{equation}\label{E:overlay}
\mathcal{H}=Im\:\partial_c\oplus Im\:\partial_j\oplus \ker_{\mathcal{H}_c} \partial_c^\ast\oplus \ker_{\mathcal{H}_j} \partial_j^\ast,
\end{equation}
where $\ker_{\mathcal{H}_c} \partial_c^\ast=\left\lbrace \omega\in\mathcal{H}_c: \partial_c^\ast \omega=0\right\rbrace$ and $\ker_{\mathcal{H}_j} \partial_j^\ast$ is defined similarly.

\begin{examples}\label{Ex:sum}\mbox{}
\begin{enumerate}
\item[(i)] Consider $X=\mathbb{R}^n$, $\mathcal{F}=H^1(\mathbb{R}^n)$ and
\[\mathcal{E}(f)=\frac12 \int_{\mathbb{R}^n} |\nabla f|^2 dx,\ \ f\in H^1(\mathbb{R}^n).\]
Then $(P_t)_{t>0}$ is the \emph{Gauss-Weierstrass semigroup} with generator $(\frac12\Delta, H^2(\mathbb{R}^n))$ and the space $\mathcal{H}=\mathcal{H}_c$ is isometrically isomorphic to $L^2(\mathbb{R}^n,\mathbb{R}^n)$. Up to this isometry, the operator $\partial$ 
coincides with the exterior derivative acting on functions, $f\mapsto df$. Interpreting the elements of $\mathcal{H}$ as vector fields, $\partial$ coincides with the gradient $f\mapsto \nabla f$ and $\partial^\ast$ agrees with minus half the divergence $v\mapsto -\frac12\diverg v$.
\item[(ii)] Let $X=\mathbb{R}^n$, $0<\alpha<2$, $\mathcal{F}=H^1(\mathbb{R}^n)$ and
\[\mathcal{E}(f)=\frac12\int_{\mathbb{R}^n} |\nabla f|^2 dx+\frac12\int_{\mathbb{R}^n}\int_{\mathbb{R}^n}\frac{|f(x)-f(y)|^2}{|x-y|^{n+\alpha}}\:dxdy,\ \ f\in H^1(\mathbb{R}^n).\]
Then the local part $\partial_cf$ of a gradient $\partial f$ may be identified with $\nabla f$ and the non-local part $\partial_jf$ with the difference operator $df$. The local part $\partial^\ast_c v$ of the divergence $\partial^\ast v$ of a vector field $v\in\mathcal{H}_a$ is represented by $-\frac12\diverg v$, and for its non-local part we have
\[\partial^\ast_jv(x)=\int_{\mathbb{R}^n}\frac{v_j(x,y)}{|x-y|^{n+\alpha}}dy.\]
\end{enumerate}
\end{examples}

Given a real valued function $a\in \Omega^1_a(\mathcal{D})$ consider the \emph{energy form $\mathcal{E}^a$ with magnetic potential $a$} given by
\[\mathcal{E}^a(f):=\lim_{t\to 0}\frac{1}{2}\int_X\int_X|f(x)-e^{ia(x,y)}f(y)|^2\Pi_t(dxdy), \ \ f\in\mathcal{C}.\]
We obtain an analog of Lemma \ref{L:boundedjumpcase}.

\begin{lemma}\label{L:boundedgeneralcase}
Let $a\in\Omega^1_a(\mathcal{D})$ be real valued. Then we have
\begin{equation}\label{E:firstestimate}
\mathcal{E}^a(f)\leq 4(\mathcal{E}(f)+\left\|f\right\|_{\sup}^2\left\|a\right\|_{\mathcal{H}}^2)
\end{equation}
and by polarization $\mathcal{E}^a$ defines a conjugate symmetric bilinear form on $\mathcal{C}$. If $\widetilde{a}\in \Omega^1_a(\mathcal{D})$ is another real valued function, then
\begin{equation}\label{E:difference}
|\mathcal{E}^a(f)-\mathcal{E}^{\widetilde{a}}(f)|\leq 4(\mathcal{E}(f)+\left\|f\right\|_{\sup}^2\left\|a\right\|_{\mathcal{H}}^2+\left\|f\right\|_{\sup}^2\left\|\widetilde{a}\right\|_{\mathcal{H}}^2)^{1/2}\left\|f\right\|_{\sup}\left\|a-\widetilde{a}\right\|_{\mathcal{H}}
\end{equation}
\end{lemma}

\begin{proof}
The estimate (\ref{E:firstestimate}) follows  as in the proof of Lemma \ref{L:boundedjumpcase}. It allows to view $\mathcal{E}^a$ as a bilinear form on $\mathcal{C}$ by polarization. The conjugate symmetry of $\mathcal{E}^a$ follows from the conjugate symmetry of the scalar products in the spaces $L^2(X\times X, \Pi_t)$, $t>0$. To see (\ref{E:difference}) note that 
\begin{multline}
|\mathcal{E}^a(f)-\mathcal{E}^{\widetilde{a}}(f)|\leq \lim_{t\to 0}\frac12\int_X\int_X\left||f(x)-e^{ia(x,y)}f(y)|-|f(x)-e^{i\widetilde{a}(x,y)}f(y)|\right|\times\\
\times \left||f(x)-e^{ia(x,y)}f(y)|+|f(x)-e^{i\widetilde{a}(x,y)}f(y)|\right|\:\Pi_t(dxdy),
\end{multline}
which by the triangle inequality and Cauchy-Schwarz does not exceed
\begin{multline}
\left(\lim_{t\to 0}\frac12\int_X\int_X|e^{i\widetilde{a}(x,y)}-e^{ia(x,y)}|^2|f(y)|^2\Pi_t(dxdy)\right)^{1/2}\times\\
\left(\lim_{t\to 0}\frac12\int_X\int_X\left|2|f(x)-f(y)|+|f(y)||e^{ia(x,y)}-1|+|f(y)||e^{i\widetilde{a}(x,y)}-1|\right|^2\Pi_t(dxdy)\right)^{1/2}.
\end{multline}
Using elementary estimates as in the proof of Lemma \ref{L:boundedjumpcase} we then arrive at (\ref{E:difference}).
\end{proof}

Recall that $\Omega^1(\mathcal{D})$ is dense in $\mathcal{H}$. We call an element $\omega$ of $\mathcal{H}$ \emph{real} if there is a sequence $(\omega_n)_n$ of real valued functions $\omega_n\in \Omega^1(\mathcal{D})$ such that $\omega=\lim_n\omega_n$ in $\mathcal{H}$. Given real $a\in\mathcal{H}_a$ we set
\[\mathcal{E}^a(f):=\lim_n\mathcal{E}^{a_n}(f), \ \ f\in\mathcal{C},\]
where $(a_n)_n$ is a sequence of real valued functions $a_n\in\Omega_a^1(\mathcal{D})$ such that $a=\lim_n a_n$ in $\mathcal{H}$. By (\ref{E:difference}) the definition of $\mathcal{E}^a$ is correct, i.e. $\mathcal{E}^a(f)$ does not depend on the choice of the sequence $(a_n)_n$. By approximation the next corollary is immediate.

\begin{corollary}\label{C:boundedgeneralcase}
Let $a, \widetilde{a}\in \mathcal{H}_a$ be real. Then $\mathcal{E}^a$ and $\mathcal{E}^{\widetilde{a}}$ satisfy the estimates (\ref{E:firstestimate}) and (\ref{E:difference}). By polarization $\mathcal{E}^a$ defines a conjugate symmetric bilinear form of $\mathcal{C}$.
\end{corollary}

We observe a new Beurling-Deny decomposition for magnetic energies.

\begin{lemma}\label{L:magneticBD}
Let $a\in \mathcal{H}_a$ be real. Then we have
\begin{equation}\label{E:Ea}
\mathcal{E}^a(f,g)=\mathcal{E}_c^a(f,g)+\mathcal{E}_j^a(f,g),\ \ f,g\in\mathcal{C},
\end{equation}
where $\mathcal{E}_c^a(f,g)=\left\langle (\partial_c+ia_c)f,(\partial_c+ia_c)g\right\rangle_{\mathcal{H}_c}$ and
\[\mathcal{E}_j^a(f,g)=\int_{X\times X\setminus\diag}(f(x)-e^{ia_j(x,y)}f(y))\overline{(g(x)-e^{ia_j(x,y)}g(y))}J(d(x,y)).\]
\end{lemma}

\begin{proof}
Suppose first that $a\in \Omega_a^1(\mathcal{D})$. The statement is a consequence of the identity
\begin{equation}\label{E:limitrel1}
\left\|(\partial_c-ia_c)f\right\|_{\mathcal{H}_c}^ 2=\lim_{\varepsilon\to 0}\lim_{t\to 0}\frac12\int\int_{\left\lbrace \varrho(x,y)\leq\varepsilon\right\rbrace}|f(x)-e^{ia(x,y)}f(y)|^2\Pi_t(dxdy)
\end{equation}
for any $f\in\mathcal{C}$. To verify (\ref{E:limitrel1}) rewrite it as
\begin{equation}\label{E:limitrel2}
\lim_{\varepsilon\to 0}\lim_{t \to 0}\frac12\int\int_{\left\lbrace \varrho(x,y)\leq \varepsilon\right\rbrace} |I_1(x,y)|^2\Pi_t(dxdy)=\lim_{\varepsilon\to 0}\lim_{t\to 0}\frac12\int\int_{\left\lbrace\varrho(x,y)\leq \varepsilon\right\rbrace}|I_2(x,y)|^2\Pi_t(dxdy),
\end{equation}
$I_1(x,y):=f(x)-f(y)-ia(x,y)f(y)$ and $I_2(x,y):=f(x)-f(y)-ia(x,y)f(y)\left(\int_0^1e^{isa(x,y)}ds\right)$. Similarly as in Lemma \ref{L:boundedjumpcase} and (\ref{E:firstestimate}) we see that 
\begin{equation}\label{E:sumI1I2}
\lim_{t\to 0}\frac12\int\int_{\left\lbrace \varrho(x,y)\leq\varepsilon\right\rbrace}|I_1(x,y)+I_2(x,y)|^ 2\Pi_t(dxdy)\leq 8\:\mathcal{E}(f)+8\left\|f\right\|_{\sup}^ 2\left\|a\right\|_{\mathcal{H}}^2.
\end{equation}
for any $\varepsilon>0$. Let $K\subset X$ be a compact set containing $\supp f$. Then $I_1$ and $I_2$ both are supported in $K\times K$. We may assume that $f$ is not identically zero. Since $a\in\Omega^1_a(\mathcal{D})$ can be seen as a continuous function on $X\times X$ vanishing on the diagonal it is uniformly continuous on $K\times K$. We may assume $\left\|a\right\|_{\mathcal{H}}>0$. Then for any $\delta>0$ there is some $\varepsilon >0$ such that $|a(x,y)|^2<\delta\left\|f\right\|_{\sup}^ {-2}\left\|a\right\|_{\mathcal{H}}^{-2}$ for all $x,y\in K$ with $\varrho(x,y)\leq\varepsilon$. For such $\varepsilon$ we have
\begin{align}\label{E:diffI1I2}
\lim_{t\to 0}\int &\int_{\left\lbrace (x,y)\in K\times K: \varrho(x,y)\leq \varepsilon\right\rbrace}|I_1(x,y)-I_2(x,y)|^2\Pi_t(dxdy)\notag\\
&=\lim_{\varepsilon\to 0}\int\int_{\left\lbrace (x,y)\in K\times K: \varrho(x,y)\leq \varepsilon\right\rbrace}|f(y)|^2|a(x,y)|^2\left(\int_0^1 (e^{isa(x,y)}-1)ds\right)^2\Pi_t(dxdy)\notag\\
&\leq \int\int_{\left\lbrace (x,y)\in K\times K: \varrho(x,y)\leq \varepsilon\right\rbrace}|f(y)|^2|a(x,y)|^4\Pi_t(dxdy)\notag\\
&\leq \delta.
\end{align}
Formulas (\ref{E:sumI1I2}) and (\ref{E:diffI1I2}) together with Cauchy-Schwarz and the trivial identity $||I_1|^2-|I_2|^2|=|I_1+I_2||I_1-I_2|$ imply (\ref{E:limitrel2}). For general nonzero real $a\in\mathcal{H}_a$ we can use the density of $\Omega^1_a(\mathcal{D})$ in $\mathcal{H}_a$ together with (\ref{E:difference}) and versions of (\ref{E:difference}) involving $\mathcal{E}^a_c$ and $\mathcal{E}^a_j$. Note that in the expression for $\mathcal{E}_j^a$ it suffices to consider the jump part $a_j$ of $a$.
\end{proof}

Given a real valued locally integrable Borel function $v$ on $X$ we also consider
\begin{equation}\label{E:electric}
\mathcal{E}^{a,v}(f,g):=\mathcal{E}^a(f,g)+\int_X f\overline{g}vd\mu, \ \ f,g\in\mathcal{C}.
\end{equation}
In Section \ref{S:Ident} we will see that under certain additional conditions on $(\mathcal{E},\mathcal{F})$, $a$ and $v$ the form $\mathcal{E}^{a,v}$ (restricted to a possibly smaller core) is closable on $L^2(X,\mu)$.

\section{Additive functionals and Nakao's theorem}\label{S:AF}

We discuss the probabilistic counterpart. Let $Y=(\Omega,\mathcal{G},\mathcal{G}_t, Y_t, \theta_t, \zeta, \mathbb{P}_x)_{x\in X}$ (in short notation $Y=(Y_t)_{t\geq 0})$) be the conservative $\mu$-symmetric Hunt process on $X$ uniquely associated with $(\mathcal{E},\mathcal{F})$ in the sense of \cite[Chapter 7]{FOT94}. Without loss of generality we may assume that $Y$ is in canonical representation. That is, the sample space $\Omega$ is the space $D([0,\infty), X_\Delta)$ of cadlag functions from $[0,+\infty)$ to $X_\Delta$, where $X_\Delta=X\cup\left\lbrace\Delta\right\rbrace$ is the one-point compactification of $X$ and the point at infinity $\Delta$ is a trap for $Y$, and for any $t\geq 0$ and $\omega\in \Omega$ we have $Y_t(\omega)=\omega(t)$. The $\sigma$-algebras $\mathcal{G}$ and $\mathcal{G}_t$ are the minimum completed $\sigma$-algebras obtained from $\mathcal{G}_\infty^0:=\sigma(Y_s: 0 \leq s<\infty)$ and $\mathcal{G}_t^0:=\sigma(Y_s: 0\leq s\leq t)$, respectively, see \cite[Section 2]{ChFiKuZh08a} or \cite[Appendix A.2]{FOT94}. By $\zeta(\omega):=\inf\left\lbrace t\geq 0: Y_t(\omega)=\Delta\right\rbrace$ we denote the lifetime of $Y$. We say that a property holds \emph{quasi-everywhere (q.e.)} on $X$ if it holds outside a set of zero $(\mathcal{E}$-capacity, see \cite[Chapter 2]{FOT94} for details. By conservativeness we have $\mathbb{P}_x(\zeta=+\infty)=1$ for q.e. $x\in X$. Recall that for any $t\geq 0$ the time shift operator $\theta_t:\Omega\to \Omega$ is defined by $(\theta_t\omega)(s):=\omega(t+s)$ for any $s\geq 0$.

As $Y$ is a Hunt process it is right continuous with left limits (c\`adl\`ag). For any $t>0$ let $\omega(t-):=\lim_{h\to 0}\omega(t-h)$ denote the left limit of $\omega$ at $t$ and define $\omega(0-)$ to be $\omega(0)$. Then $(\omega(t-))_{t\geq 0}$ is left-continuous. More generally, given a stochastic process $Z=(Z_t)_{t\geq 0}$ on $\Omega$ (with values in a metric space) we set $Z_{t-}(\omega):=\lim_{h\to 0} Z_{t-h}(\omega)$, $\omega\in \Omega$,
and $Z_{0-}:=Z_0$. By construction the process $(Z_{t-})_{t\geq 0}$ is left-continuous. Applied to $Y$ these contructions are consistent, i.e. $Y_{t-}(\omega)=\omega_{t-}$ for all $t$.

Given a measure $m$ on $X$ we write $\mathbb{P}_m(A)=\int_X \mathbb{P}_x(A)m(dx)$, and for a random variable $Z$ we write $\mathbb{E}_m[Z]=\int_X \mathbb{E}_x[Z]m(dx)$, where $\mathbb{E}_x$ is the expectation with respect to $\mathbb{P}_x$.

A process $A=(A_t)_{t\geq 0}$ is an \emph{additive functional (AF)} of $Y$ (in the sense of \cite[Chapter 5]{FOT94}) if $A_t$ is $\mathcal{G}_t$-measurable for all $t\geq 0$ and there are a set $D\in \mathcal{G}_\infty$ and an exeptional set $N\subset X$ such that the following conditions are satisfied: For any $x\in X\setminus N$ we have $\mathbb{P}_x(D)=1$, for all $t>0$ we have $\theta_tD\subset D$, and for any $\omega\in D$ the function $t\mapsto A_t(\omega)$ is c\`adl\`ag, $A_0(\omega)=0$, $|A_t(\omega)|<+\infty$ for all $t$, and 
\begin{equation}\label{E:additivity}
A_{t+s}(\omega)=A_s(\omega)+A_t(\theta_s\omega)
\end{equation}
for all $s,t\geq 0$.

Every function $f\in\mathcal{F}$ has a quasi-continuous representant $\widetilde{f}$,  \cite[Section 2]{FOT94}. To simplify notation we write $f$ with the silent agreement to always work with $\widetilde{f}$. \emph{Fukushima's theorem}, \cite[Theorem 5.2.2]{FOT94}, states that for the AF $A^f=(A^f_t)_{t\geq 0}$ defined by $A^f_t=f(Y_t)-f(Y_0)$ we have the unique decomposition
\begin{equation}\label{E:Fukushimadecomp}
A^f=M^f+N^f\ \ ,\ \ t\geq 0.
\end{equation}
$\mathbb{P}^x$-a.s. for q.e. $x\in X$, where $M^f=(M_t^f)_{t\geq 0}$ is a square integrable $\mathbb{P}^x$-martingale and $N^f=(N^f_t)_{t\geq 0}$ is a continuous AF of zero energy. More precisely, with the \emph{($\mu$-)energy} of an AF $A=(A_t)_{t\geq 0}$ of $Y$ defined by \[\mathbf{e}(A):=\lim_{t\to 0}\frac{1}{2t}\mathbb{E}_\mu(A_t^2),\]
\begin{multline}
\mathring{\mathcal{M}}:=\left\lbrace M: \text{$M$ finite cadlag AF of $Y$ with $\mathbf{e}(M)<+\infty$}\right.\notag\\
\left. \text{such that for each $t>0$ we have $\mathbb{E}_x(M_t^2)<+\infty$ and $\mathbb{E}_x(M_t)=0$ for q.e. $x\in X$} \right\rbrace
\end{multline}  
denoting the \emph{space of martingale AF's of finite energy} and 
\begin{multline}
\mathcal{N}_c:=\left\lbrace N: \text{$N$ finite continuous AF of $Y$ with $\mathbf{e}(N)=0$}\right.\notag\\
\left. \text{and such that $\mathbb{E}_x(|N_t|)<+\infty$ q.e. for each $t>0$}\right\rbrace  
\end{multline}
the \emph{space of continuous AF's of zero energy}, we have $M^f\in\mathring{\mathcal{M}}$ and $N^f\in\mathcal{N}_c$ in (\ref{E:Fukushimadecomp}). For $f\in \dom L$ we observe $N^f_t=\int_0^t (Lf)(Y_s)ds$ and (\ref{E:Fukushimadecomp}) is a semimartingale decomposition with respect to $\mathbb{P}_x$ for q.e. $x\in X$. Polarizing the energy $\mathbf{e}$ we obtain a bilinear form that turns $\mathring{\mathcal{M}}$ into a Hilbert space $(\mathring{\mathcal{M}},\mathbf{e})$. Given $M, N\in\mathring{\mathcal{M}}$ let $\left\langle M, N\right\rangle$ denote their sharp bracket and $\mu_{\left\langle M, N\right\rangle}$ the (signed) Revuz measure of $\left\langle M, N\right\rangle$. We write $\mu_{\left\langle M\right\rangle}$ for $\mu_{\left\langle M, M\right\rangle}$. For a martingale AF of form $M^f$ as in (\ref{E:Fukushimadecomp}) with $f\in\mathcal{C}$ we observe
$\mu_{\left\langle M^f\right\rangle}=2\Gamma(f)$. For $g\in L_2(X,\mu_{\left\langle M\right\rangle})$ the stochastic integral $g\bullet M\in\mathring{\mathcal{M}}$ of $g$ with respect to $M$ is defined by the identity
\[\mathbf{e}(g\bullet M, N)=\frac12 \int_X g\mu_{\left\langle M, N\right\rangle}, \ \ N\in\mathring{\mathcal{M}}.\]
For $g\in \mathcal{C}$ and each $t>0$ we have  
\[(g\bullet M)_t=\int_0^t g(Y_{s-})dM_s, \ \]
$\mathbb{P}_x$-a.s. for q.e. $x\in X$, \cite[Lemma 5.6.2]{FOT94} or \cite[Lemma 2.3]{Ku10}. Here the right hand side may be interpreted as a usual stochastic integral of a predictable integrand with respect to a square integrable martingale. Recall that the $(\mathcal{G}_t)_{t\geq 0}$-predictable $\sigma$-algebra is the smallest $\sigma$-algebra on $[0,\infty)\times \Omega$ containing all $\mathbb{P}_\nu(\mathcal{G})$-evanescent sets for all probability measures $\nu$ on $X\cup \left\lbrace \Delta\right\rbrace$ and with respect to which all $(\mathcal{G}_t)_{t\geq 0}$-adapted c\`agl\`ad (left continuous with right limits) processes are measurable.

Given $f\otimes g \in\mathcal{C}\otimes \mathcal{C}$ put
\[\Theta(f\otimes g):=g\bullet M^f,\]
where $M^f\in\mathring{\mathcal{M}}$ is the martingale additive functional in (\ref{E:Fukushimadecomp}). Since $\Theta$ is a linear map and $\mathbf{e}(\Theta(f\otimes g))=\left\|f\otimes g\right\|_{\mathcal{H}}^2$ we can extend $\Theta$ to an isometry of $\mathcal{C}\otimes\mathcal{C}$ into $\mathring{\mathcal{M}}$. 

\begin{theorem}\label{T:Nakao}
The map $\Theta$ extends to an isometric isomorphism of $\mathcal{H}$ into $\mathring{\mathcal{M}}$ and for $\omega\in\mathcal{H}$ and $M=\Theta(\omega)$ we have $\mu_{\left\langle M\right\rangle}=2\Gamma_{\mathcal{H}}(\omega)$. 
\end{theorem}

Nakao proved this theorem in \cite{N85} for diffusions on manifolds, following earlier work of Ikeda, Manabe and Watanabe \cite{IM79, IW81}. In \cite[Theorem 9.1]{HRT} we obtained Theorem \ref{T:Nakao} for general symmetric Hunt processes on locally compact separable metric spaces as a byproduct of the approach of Cipriani and Sauvageot \cite{CS03}.

Set $\mathring{\mathcal{M}}_\partial:=\left\lbrace M^f: f\in\mathcal{F}\right\rbrace$ and
denote by $\mathring{\mathcal{M}}_c$ the closed subspace of $\mathring{\mathcal{M}}$ spanned by the continuous martingale AF's of finite energy and let $\mathring{\mathcal{M}}_j$ denote its orthogonal complement. The first statement in the next lemma is obvious, the second follows from \cite[Lemma 5.3.3]{FOT94}.

\begin{corollary}\label{C:Nakao}
The image of $Im\:\partial$ under $\Theta$ equals $\mathring{\mathcal{M}}_\partial$. Therefore $\mathring{\mathcal{M}}_\partial$ is a closed subspace of $\mathring{\mathcal{M}}$. The images of $\mathcal{H}_c$ and $\mathcal{H}_j$ are $\mathring{\mathcal{M}}_c$ and $\mathring{\mathcal{M}}_j$, respectively.
\end{corollary}

Together with (\ref{E:BDH}) this implies that any element $M\in\mathring{\mathcal{H}}$ may uniquely be written as an orthogonal sum
\begin{equation}\label{E:overlay2}
M=M_{\partial,c}+M_{\partial, j}+M_{\bot,c}+M_{\bot,j},
\end{equation}
where $M_{\partial, c}\in \mathring{\mathcal{M}}_\partial\cap \mathring{\mathcal{M}}_c$ and $M_{\partial,j} \in 
\mathring{\mathcal{M}}_\partial\cap \mathring{\mathcal{M}}_j$ and the remaining summands 
$M_{\bot,c}$ and $M_{\bot,j}$ are the projections of $M$ onto the complements of $\mathring{\mathcal{M}}_\partial\cap \mathring{\mathcal{M}}_c$ in $\mathring{\mathcal{M}}_c$
and $\mathring{\mathcal{M}}_\partial\cap \mathring{\mathcal{M}}_j$ in $\mathring{\mathcal{M}}_j$, respectively. This is (\ref{E:overlay}), rewritten for AF's.

\section{Divergence functionals and Stratonovich line integrals}\label{S:lineintegral}

Let $\mathcal{N}_c^\ast$ denote the space of continuous AF's $N=(N_t)_{t\geq 0}$ of $Y$ of the form $N_t=N^f_t+\int_0^tg(Y_s)ds$ for some functions $f\in \mathcal{F}$ and $g\in L_2(X,\mu)$. In \cite{N85} Nakao constructed a linear operator $\Lambda:\mathring{\mathcal{M}}\to \mathcal{N}_c^\ast$ by 
\[\Lambda(M)_t=N^u_t-\int_0^t u(Y_s)ds\]
for $M\in\mathring{\mathcal{M}}$, where $u$ is the unique element of $\mathcal{F}$ such that 
$\mu_{\left\langle M^h, M\right\rangle}(X)=\mathcal{E}_1(u,h)$,  $h\in\mathcal{F}$.
To $\Lambda$ one usually refers as \emph{Nakao's divergence operator}. The AF $\Lambda(M)$ is characterized by the limit relation
\[\lim_{t\to 0}\frac{1}{t}\mathbb{E}_{h\mu}[\Lambda(M)_t]=-\frac12\mu_{\left\langle M^h, M\right\rangle}(X), \ \ h\in\mathcal{F}.\]
By Theorem \ref{T:Nakao} and Corollary \ref{C:Nakao} the decomposition (\ref{E:Hodge}) of $\mathcal{H}$ induces an orthogonal decomposition of $\mathring{\mathcal{M}}$ into $\mathring{\mathcal{M}}_\partial$ and its complement, the kernel of $\Lambda$.

\begin{corollary}\label{C:Lambdagrad}
The image of $\ker\:\partial^\ast$ under $\Theta$ is $\ker\:\Lambda$ and consequently $\mathring{\mathcal{M}}=\mathring{\mathcal{M}}_\partial\oplus \ker\:\Lambda$.
In particular, we have $\Lambda(M)=\Lambda(M_\partial)$, where $M_\partial$ is the projection of $M$ onto $\mathring{\mathcal{M}}_{\partial}$.
\end{corollary}

For local and transient Dirichlet spaces a similar statement was proved in \cite[Section 3]{Fi93}. Corollary \ref{C:Lambdagrad} is valid also in the non-local case.

\begin{proof}
Let $\omega\in\mathcal{H}$ and $M=\Theta(\omega)$. Let $\mu_{\Lambda(M)}$ denote the signed Revuz measure of the continuous additive functional $\Lambda(M)$. By \cite[Corollary 9.3]{HRT} we have $(-\partial^\ast \omega)(h)=\int_Xh d\mu_{\Lambda(M)}$,  $h\in\mathcal{C}$.
Consequently $\Lambda(M)$ is zero the zero functional if and only if $\partial^\ast \omega=0$.
\end{proof}

\begin{remark}\label{R:cosycase}
If $\omega\in\mathcal{H}$ is such that $\omega=\partial u+\eta$ with $\eta\in \ker \partial^\ast$ and $u\in\dom L$, then $M=\Theta(\omega)$ satisfies $\Lambda(M)_t=\int_0^t (Lu)(Y_s)ds$.
In this case $\Lambda(M)$ is a continuous AF of bounded variation and
$-\partial^\ast \omega=Lu$ is the density of the (signed) Revuz measure $\mu_{\Lambda(M)}$ of $\Lambda(M)$.
\end{remark}

In \cite{IM79} Ikeda and Manabe defined Stratonovich line integrals of $C^2$-differential $1$-forms along paths of Brownian motion on a Riemannian manifold, \cite[Definition 2.1]{IM79}. Nakao \cite{N85} generalized this for diffusions on manifolds, \cite[Definitions 3.4 and 5.1]{N85}. These definitions used local coordinates. In \cite[Theorem 5.2]{N85} he provided a coordinate free expression, which we now use to define the Stratonovich line integral.

\begin{definition}\label{D:Stratonovich}
Let $\omega\in\mathcal{H}$ be real valued. For any $t\geq 0$ the \emph{stochastic line integral of $\omega$ along $Y([0,t])$} is defined by
\[\int_{Y([0,t])}\omega:=\Theta(\omega)_t+\Lambda(\Theta(\omega))_t.\]
\end{definition}

\begin{remark}\label{R:cosycase2}\mbox{}
\begin{enumerate}
\item[(i)] For $\omega=g\partial f$ we obtain $\int_{Y([0,t])}\omega=g\bullet M^f+\Lambda(g\bullet M^f)_t$, which for diffusions on manifolds agrees with Nakao's definition, \cite[Definition 3.4]{N85}.
\item[(ii)]  In \cite{Ku10} Kuwae defined It\^o and Stratonovich line integrals of suitable functions with respect to Dirichlet processes, \cite[Definition 4.1]{Ku10}. His results are based on a generalization of Nakao's functional $\Lambda$ by Chen, Fitzsimmons, Kuwae and Zhang, \cite[Definition 3.3]{ChFiKuZh08a}, which is able to deal with martingale AF's that are only locally square integrable. Technically the probabilistic interpretation of Definition \ref{D:Stratonovich} might be viewed as a special case of the Stratonovich integrals in \cite[Definition 4.1]{Ku10} if the constant $\mathbf{1}$ is integrated. However, differential $1$-forms and their line integrals are neither discussed in \cite{Ku10} nor in \cite{ChFiKuZh08a}.
\item[(iii)] Recall Remark \ref{R:cosycase}. For real valued $\omega\in\mathcal{H}$ of form $\omega=\partial u+\eta$ with $u\in \dom L$ and $\eta\in \ker \partial^\ast$ we observe that
\begin{equation}\label{E:genDynkin}
\int_{Y([0,t)])}\omega=\Theta(\omega)_t+\int_0^t Lu (Y_s)ds,
\end{equation}
which in this case is a (c\`adl\`ag) $\mathbb{P}_x$-semimartingale for q.e. $x\in X$.
\end{enumerate}
\end{remark}

\section{Time reversal and jump functions}\label{S:timereversal}

It is well known that for diffusions the Stratonovich integral is antisymmetric under time reversal, \cite{Simon79} and \cite{Fi93}. We give a proof for general conservative regular Dirichlet forms and observe some connections between purely discontinuous AF's and  differential $1$-forms.

For $t\geq 0$ the time reversal operator $r_t:\Omega\to \Omega$ is defined by
\[r_t(\omega)(s):=\begin{cases} \omega((t-s)-) \ &\text{ for $0\leq s\leq t$ and}\\
\omega(0)\ &\text{ for $s\geq t$},\end{cases}\] 
recall that for any $t>0$ $\omega(t-):=\lim_{h\to 0}\omega(t-h)$ is the left limit of $\omega$ at $t$ and $\omega(0-)$ is defined to be $\omega(0)$. Following \cite{ChFiKuZh08a} and \cite{Fi93} we call an AF $A=(A_t)_{t\geq 0}$ of $Y$ \emph{even} if $A_t\circ r_t=A_t$ $\mathbb{P}_\mu$-a.e. for each $t>0$ and \emph{odd} if $A_t\circ r_t=-A_t$ $\mathbb{P}_\mu$-a.e. for each $t>0$. Each AF $A=(A_t)_{t\geq 0}$ may uniquely be written as the sum of its even part, given by
$A^{even}_t=\frac12(A_t+A_t\circ r_t)$ and its odd part, given by $A^{odd}_t=\frac12(A_t-A_t\circ r_t)$. See \cite{Fi93}.

\begin{theorem}\label{T:antisymm}
Let $a\in\mathcal{H}_a$ be real. The Stratonovich line integral $S=(S_t)_{t\geq 0}$, of $a$, given by
\[S_t:=\int_{Y[0,t])}a,\ \ t>0,\]
agrees with the odd part of $\Theta(a)$. The even part of $\Theta(a)$ is $-\Lambda(\Theta(a))$.
\end{theorem}

For the strongly local case Theorem \ref{T:antisymm} was proved by Fitzsimmons in \cite[Theorem 3.1 and Corollary 3.1]{Fi93}, it also follows from \cite{LZh94}. For symmetric Hunt processes with nontrivial jump part Theorem \ref{T:antisymm} seems to be new. In \cite[Theorem 2.18 and Remark 3.4 (ii)]{ChFiKuZh08a}
it is shown that for general regular Dirichlet forms and under some integrability conditions on $M\in\mathring{\mathcal{M}}$ the AF $\Lambda(M)$ is continuous and even. Theorem \ref{T:antisymm} is a consequence of these results and we sketch this conclusion. 

For any finite c\`adl\`ag AF $M=(M_t)_{t\geq 0}$ of $Y$ there exists a Borel function $\varphi$ on the product space $X\times X$ vanishing on the diagonal, $\varphi(x,x)=0$, $x\in X$, and such that 
\[M_t-M_{t-}=\varphi(Y_{t-},Y_t) \ \text{ for every $t>0 $ \ \ $\mathbb{P}_\mu$-a.e.}\]
This function $\varphi$ is uniquely determined $J$-a.e. and usually referred to as \emph{the jump function of $M$}. See \cite[formula (1.8)]{ChFiKuZh08a} and \cite[Lemma 3.2]{ChFiTaYiZh}. By definition $M_c$ has jump function zero, hence the jump function depends only on $M_j$ (which may replace $M$ in the above identity). Using Theorem \ref{T:Nakao}  we can identify the jump function of $M\in\mathring{\mathcal{M}}$ as the jump part $\omega_j=\Theta^{-1}(M_j)$ of the $1$-form $\omega=\Theta^{-1}(M)$. 

\begin{theorem}\label{T:jumpfunction}
Let $\omega\in\mathcal{H}$ and $M=\Theta(\omega)$. Then the jump function of $M$ is $\omega_j$, i.e. 
\begin{equation}\label{E:jumpfct}
M_t-M_{t-}=\omega_j(Y_{t-},Y_t) \ \text{ for every $t>0 $  $\mathbb{P}_\mu$-a.e.}
\end{equation}
\end{theorem}

To prove Theorem \ref{T:jumpfunction} we use the \emph{L\'evy system formula}. A pair $(N,H)$ is called a \emph{L\'evy system} for $Y$ if $N=N(x,dy)$ is a kernel on $(X,\mathcal{B}(X))$ with $N(x,\left\lbrace x\right\rbrace )=0$ for any $x\in X$ and $H$ is a positive continuous AF of $Y$ such that
for any $(\mathcal{G}_t)_{t\geq 0}$-predictable process $(Z_t)_{t\geq 0}$, any nonnegative Borel function $\varphi$ on $X\times X$ vanishing on $\diag$ and any $x\in X$ we have
\begin{equation}\label{E:LSF}
\mathbb{E}_x\left[\sum_{0<s\leq t} Z_s\varphi(Y_{s-},Y_s)\right]=\mathbb{E}_x\left[\int_0^ tZ_s\int_X\varphi(Y_s,y)N(Y_s,dy)dH_s\right]
\end{equation}
Formula (\ref{E:LSF}) is equivalent to its special case for $Z\equiv 1$, see \cite[p. 437]{ChFu12}, \cite[p. 346]{Sh88} or \cite{BJ73}. 

\begin{proof}
For $f,g\in\mathcal{C}$ the jump function of $g\bullet M^f$ is given by $g(x)(f(x)-f(y))=gdf(x,y)$, i.e. we have
\[(g\bullet M^f)_t-(g\bullet M^f)_{t-}=g(Y_{t-})(f(Y_{t-})-f(Y_t)), \ \ t>0,\]
$\mathbb{P}_\mu$-a.s. See the proof of \cite[Theorem 3.6]{ChFiKuZh08a}. For general $\omega\in \mathcal{H}$ let $\omega^n:=f_n\otimes g_n$ with $f_n, g_n\in\mathcal{C}$ be such that $(\omega^n)_n$ approximates $\omega\in\mathcal{H}$. By projection clearly also $\lim_n \omega_j^n=\omega_j$ in $L^2(X\times X\setminus \diag, J)$ and by Theorem \ref{T:Nakao} the martingale AF $M=\Theta(\omega)$ is approximated in $\mathring{\mathcal{M}}$ by the stochastic integrals $M_n:=g_n\bullet M^{f_n}$ (see also \cite[Lemma 5.6.3]{FOT94}). To pass to the limit we follow the arguments of \cite[Theorem 5.2.1]{FOT94}. For any $n$ the process $M-M^n$ is a square integrable $\mathbb{P}_\mu$-martingale, hence $\mathbb{P}_\mu(\sup_{0\leq s\leq T} |M_s-M_s^n|<\varepsilon)\leq \frac{4T}{\varepsilon}\left\|\omega-\omega^n\right\|_{\mathcal{H}}^2$ 
for any $T>0$ and $\varepsilon>0$. Now let $(\omega^{n_k})_k$ be a subsequence such that 
\begin{equation}\label{E:squeeze}
\left\|\omega-\omega^{n_k}\right\|_{\mathcal{H}}<2^{-k}
\end{equation}
for all $k$. Then \v{C}ebyshev's inequality yields $\mathbb{P}_\mu(\sup_{0\leq s\leq t} |M_s-M_s^{n_k}|>2^{-k})\leq 2^{-n+2}T$
and by Borel-Cantelli there exists $\Omega_0\in\mathcal{F}$ with $\mathbb{P}_\mu(\Omega_0)=1$ such that $\lim_k\sup_{0\leq s\leq T}|M_s-M_s^{n_k}|=0$ on $\Omega_0$ for all $T>0$. This implies 
\begin{equation}\label{E:firstlim}
M_t-M_{t-}=\lim_k M_t^{n_k}-M_{t-}^{n_k}=\lim_k \omega^{n_k}(Y_{t-},Y_t)
\end{equation}
for all $t>0$ $\mathbb{P}_\mu$-a.s. On the other hand 
\begin{align}
\mathbb{E}_\mu\left[\sum_{s\leq T} (\omega_j-\omega^{n_k})^2(Y_{s-},Y_s)\right]&
=\mathbb{E}_\mu\left[\int_0^T\int_X(\omega_j-\omega^{n_k})^2(Y_s,y)N(Y_s,dy)dH_s\right]\notag\\
&=T\int_X\int_X(\omega_j-\omega^{n_k})^2(x,y)N(x,dy)\mu_H(dx)\notag\\
&=T\int_{X\times X\setminus \diag}(\omega_j-\omega^{n_k})^2(x,y)J(d(x,y)) \notag\\
&\leq T\left\|\omega-\omega^{n_k}\right\|_{\mathcal{H}}^2\notag
\end{align}
for any $T>0$, where we have used (\ref{E:LSF}) and the fact that $Y$ is conservative. Using (\ref{E:squeeze}) we observe
\[\mathbb{P}_\mu\left(\sum_{s\leq T}(\omega_j-\omega^{n_k})^2(Y_{t-},Y_t)>2^{-k}\right)\leq 2^{-k}T\]
for all $T>0$, and again we can find $\Omega_1\in\mathcal{F}$ with $\mathbb{P}_\mu(\Omega_1)=1$ such that for all $T>0$ we have $\lim_k \sum_{s\leq T} (\omega_j-\omega^{n_k})^2(Y_{t-},Y_t)=0$ on $\Omega_1$. This implies $a_j(Y_{t-},Y_t)=\lim_k\omega^{n_k}(Y_{t-}, Y_t)$ for all $t>0$, $\mathbb{P}_\mu$-a.s. With (\ref{E:firstlim}) we obtain (\ref{E:jumpfct}). 
\end{proof}

Theorem \ref{T:jumpfunction}, \cite[Definition 3.3]{ChFiKuZh08a} and \cite[Theorem 3.6]{ChFiKuZh08a} imply the following.

\begin{corollary}\label{C:simplified}
Let $\mathcal{H}_a$ be real and set $M=\Theta(a)$. We have $\Lambda(M)_t=-\frac12(M_t+M_t\circ r_t+a_j(Y_t, Y_{t-}))$ for any $t>0$ 
$\mathbb{P}_\mu$-a.s.
\end{corollary}

To see Theorem \ref{T:antisymm} note that for fixed $t>0$ we have $Y_{t-}=Y_t$ $\mathbb{P}_\mu$-a.s. and since $a_j$ vanishes on the diagonal, $a_j(Y_{t-},Y_t)=0$. Consequently for any fixed $t>0$, 
\[S_t=\frac12\Theta(a)_t-\frac12\Theta(a)_t\circ r_t\ \ \text{$\mathbb{P}_\mu$-a.s.}\]

With Corollary \ref{C:Nakao} we observe generalization of a statement from \cite[Corollary 3.1]{Fi93}.

\begin{corollary}
For real $a\in \mathcal{H}_a$ the AF $\Theta(a)$ is odd if and only if $\partial^\ast a=0$. 
\end{corollary}

We need a version of a well known representation for the discontinuous parts of martingale AF's as limits of compensated sums, \cite[p. 935]{ChFiKuZh08a} or \cite[Section 5.3]{FOT94}, the proof is similar.

\begin{lemma}\label{L:folklore}
Assume that the jump measure $J$ has a kernel, $J(d(x,y))=\frac12 n(x,dy)\mu(dx)$. Let $a\in\mathcal{H}_a$ be real and set $M=\Theta(a)$. For any $t>0$ we have
\[M_t^j=\lim_{\varepsilon\to 0}\left\lbrace \sum_{0<s\leq t}a_j(Y_s,Y_{s-})\mathbf{1}_{\left\lbrace |a_j(Y_s,Y_{s-})|>\varepsilon\right\rbrace }
-\int_0^t\int_{\left\lbrace y\in X: |a_j(y,Y_s)|>\varepsilon\right\rbrace}a_j(y,Y_s)n(Y_s,dy)ds   \right\rbrace,\]
the limit taken in $L^2(\mathbb{P}_\mu)$.
\end{lemma}

In Lemma \ref{L:folklore} $N(x,dy)=n(x,dy)$ and $H(t)=t$ provide a L\'evy system $(N,H)$ for $Y$.

\begin{examples}\label{Ex:alphastablecase}\mbox{}
\begin{enumerate}
\item[(i)] Let $X=\mathbb{R}^n$, $0<\alpha<2$, and let $Y=(Y_t)_{t\geq 0}$ the isotropic $\alpha$-stable L\'evy process on $\mathbb{R}^n$. Moreover, let $a\in L^2_a(\mathbb{R}^n\times\mathbb{R}^n\setminus \diag, \frac12|x-y|^{-n-\alpha}dxdy)$ be real and such that $\partial^\ast a$ is in $L^2(\mathbb{R}^n)$. Then
\[\Lambda_t(M)=\int_0^t \int_{\mathbb{R}^n}\frac{a(y,Y_s)}{|Y_s-y|^{n+\alpha}}dyds.\]
According to Lemma \ref{L:folklore} we have
\[S_t=\lim_{\varepsilon\to 0}\left\lbrace \sum_{0<s\leq t}a(Y_s,Y_{s-})\mathbf{1}_{\left\lbrace |a(Y_s,Y_{s-})|>\varepsilon\right\rbrace }
+\int_0^t\int_{\left\lbrace y\in X: |a(y,Y_s)|\leq\varepsilon\right\rbrace}\frac{a(y,Y_s)}{|Y_s-y|^{n+\alpha}}dyds\right\rbrace.\]
\item[(ii)] In the situation of Examples \ref{Ex:sum} (ii) the associated process is the sum $B_t+Y_t$ of an $n$-dimensional Brownian motion $B=(B_t)_{t>0}$ and an isotropic $\alpha$-stable L\'evy process $Y=(Y_t)_{t>0}$ that are independent under $\mathbb{P}_x$ for q.e. $x\in X$. We obtain
\begin{multline}
S_t=\int_0^t a_c(B_s)\circ dB_s \notag\\
+\lim_{\varepsilon\to 0}\left\lbrace \sum_{0<s\leq t}a_j(Y_s,Y_{s-})\mathbf{1}_{\left\lbrace |a_j(Y_s,Y_{s-})|>\varepsilon\right\rbrace }
+\int_0^t\int_{\left\lbrace y\in X: |a_j(y,Y_s)|\leq\varepsilon\right\rbrace}\frac{a_j(y,Y_s)}{|Y_s-y|^{n+\alpha}}dyds\right\rbrace.\notag
\end{multline}
\item[(iii)] If $(V,b,\mu)$ is a weighted graph then (\ref{E:graphform}) is closable on $L^2(V,\mu)$. Let $Y=(Y_t)_{t\geq 0}$ be the associated continuous time Markov chain on $V$. If $a\in\mathcal{H}_a$ is 
bounded, then $\sum_{y\in V} |a(y,x)|b(x,y)<+\infty$, and as the number of jumps in a compact interval is finite, we have $\int_0^t \sum_{y\in V} |a(y,Y_s)|b(y,Y_s)ds<+\infty$. This implies $S_t=\sum_{0<s\leq t} a(Y_s,Y_{s-})$,
what recovers \cite[Definition 3.1]{GKS}.
\end{enumerate}
\end{examples}

\section{Feynman-Kac-It\^o formula}\label{S:FKI}

Suppose $a\in\mathcal{H}_a$ is real valued and $v$ is a real valued Borel function. For $t>0$ and bounded Borel $f$ set
\begin{equation}\label{E:FKI}
P_t^{a,v}f(x):=\mathbb{E}_x[e^{i\int_{Y([0,t])}a - \int_0^t v(Y_s)ds}f(Y_t)], \ \ x\in X.
\end{equation}
Theorem \ref{T:semigroup} tells that the Feynman-Kac-It\^o type formula (\ref{E:FKI}) defines a semigroup on $L^2(X,\mu)$. Given a real valued function $v$, let $v_-:=- (u\wedge 0)$ denote its negative part. 

\begin{theorem}\label{T:semigroup}
Let $a\in\mathcal{H}_a$ be real and let $v$ be a real valued Borel function such that $v_-$ is uniformly bounded. For any $t>0$ the operator $P^{a,v}_t$ extends to a bounded linear operator on $L^2(X,\mu)$ satisfying
\begin{equation}\label{E:L2bound}
\left\|P_t^{a,v}f\right\|_{L^2(X,\mu)}\leq e^{t\left\|v_-\right\|_{\sup}}\left\|f\right\|_{L^2(X,\mu)}, \ \ f\in L^2(X,\mu),
\end{equation}
and the family $(P^{a,v}_t)_{t>0}$ is a strongly continuous semigroup of bounded self-adjoint operators on $L^2(X,\mu)$. Moreover, for any $t>0$ the operator $P_t^{a,v}$ extends to a bounded linear operator on $L^1(X,\mu)$ with 
\[\left\|P_t^{a,v}f\right\|_{L^1(X,\mu)}\leq e^{t\left\|v_-\right\|_{\sup}}\left\|f\right\|_{L^1(X,\mu)},\ \ f\in L^1(X,\mu),\]
and $(P_t^{a,v})_{t>0}$ is a strongly continuous semigroup of bounded linear operators on $L^1(X,\mu)$.
\end{theorem}

As $a\in\mathcal{H}_a$ is fixed, we use again the abbreviation $S_t=\int_{Y([0,t])}a$. 

\begin{proof}
The estimate (\ref{E:L2bound}) follows from
\begin{align}\label{E:L2bound2}
\int_X |\mathbb{E}_x[e^{iS_t-\int_0^tv(Y_s)ds}f(Y_t)]|^2\mu(dx)&\leq \int_X\mathbb{E}_x[|e^{iS_t-\int_0^tv(Y_s)ds}|^2|f(Y_t)|^2]\mu(dx)\\
&\leq e^{2t\left\|v_-\right\|_{\sup}}\int_X P_t(|f|^2)(x)\mu(dx)\notag\\
&\leq e^{2t\left\|v_-\right\|_{\sup}}\int_X|f(x)|^2\mu(dx),\notag
\end{align}
note that any $P_t$ is also contractive on $L^1(X,\mu)$. For any $s,t>0$ and $\mu$-a.e. $x\in X$ we have 
\begin{align}
P^{a,v}_t(P^{a,v}_sf)(x)&=\mathbb{E}_x[e^{iS_t-\int_0^tv(Y_r)dr}(P^{a,v}_sf)(Y_t)]\notag\\
&=\mathbb{E}_x[e^{iS_t-\int_0^tv(Y_r)dr}\mathbb{E}_{Y_s}[e^{iS_s-\int_0^s v(Y_r)dr}f(Y_s)]]\notag\\
&=\mathbb{E}_x[e^{iS_t+iS_s(\theta_t)-\int_0^tv(Y_r)dr-\int_0^s v(Y_{t+r})dr}f(Y_{t+s})]\notag\\
&=\mathbb{E}_x[e^{iS_{t+s}-\int_0^{t+s}v(Y_r)dr}f(Y_{t+s})]\notag\\
&=P^{a,v}_{t+s}f(x)\notag
\end{align}
by the Markov property and additivity (\ref{E:additivity}). The strong continuity follows from
\begin{align}\label{E:strongcont}
\int_X|P_t^{a,v}f(x)&-f(x)|^2\mu(dx)=\int_X|\mathbb{E}_x[e^{iS_t-\int_0^tv(Y_r)dr}f(Y_t)-f(x)]|^2\mu(dx)\\
&\leq 4\int_X|\mathbb{E}_x[e^{iS_t-\int_0^tv(Y_r)dr}f(Y_t)-f(Y_t)]|^2\mu(dx)+4\int_X|\mathbb{E}_x[f(Y_t)-f(x)]|^2\mu(dx)\notag\\
&\leq 4\int_X\mathbb{E}_x|e^{iS_t-\int_0^tv(Y_r)dr}-1|^2|f(x)|^2\mu(dx)+4\int_X|P_tf(x)-f(x)|^2\mu(dx),\notag
\end{align}
because the first summand is bounded by $4e^{2t\left\|v_-\right\|_{\sup}}\int_X\mathbb{E}_x[iS_t-\int_0^tv(Y_r)dy]^2|f(x)|^2\mu(dx)$,
what vanishes as $t$ goes to zero due to the cadlag property of $(S_t)_{t\geq 0}$, and the second summand vanishes by the strong continuity of $(P_t)_{t>0}$. To see the symmetry of the operators $P_t^{a,v}$ note that by the $\mu$-symmetry and conservativeness of $(P_t)_{t\geq 0}$ we have $\mathbb{E}_\mu[F\circ r_t]=\mathbb{E}_\mu[F]$
for any $t>0$ and any $\mathcal{F}_t$-measurable function $F$, see \cite[Lemma 2.1]{Fi93}. 
On the other hand, $t\mapsto \int_0^t v(Y_s)ds$ is an even AF and by Theorem \ref{T:antisymm}
$(C_t)_{t\geq 0}$ is odd. Combining, 
\[\mathbb{E}_\mu[e^{iS_t-\int_0^t v(Y_s)ds}f(Y_t)\overline{g(Y_0)}]\notag\\
=\mathbb{E}_\mu[e^{-iS_t-\int_0^t v(Y_s)ds} f(Y_0)\overline{g(Y_t)}]\]
for any fixed $t>0$. The $L^1(X,\mu)$-bound and the continuity on $L^1(X,\mu)$ follow similarly as in (\ref{E:L2bound2}) and (\ref{E:strongcont}).
\end{proof}

For two functions $f\in L^1(X,\mu)$ and $g\in L^\infty(X,\mu)$ we write $\left\langle u,v\right\rangle:=\int_X f\overline{g}d\mu$ to denote the $L^1$-$L^\infty$ dual pairing.
We may replace $S_t$ in (\ref{E:FKI}) by $S_{t-}$ because for any fixed $t>0$ we have $Y_{t-}=Y_t$ $\mathbb{P}_\mu$-a.s. and $S_t-S_{t-}=a_j(Y_{t-},Y_t)$ with $a_j$ vanishing on the diagonal.

\begin{lemma}\label{L:replace}
For any $t>0$ and any $f\in L^1(X,\mu)$ and $g\in L^\infty(X,\mu)$ we have 
\[\left\langle P_t^ {a,v}f,g\right\rangle=\int_X \mathbb{E}_x[e^{iS_{t-}-\int_0^ t v(Y_s)ds}f(Y_s)] \overline{g(x)}\mu(dx).\]
Similarly if $f,g\in L^2(X,\mu)$.
\end{lemma}

\begin{examples}  For the isotropic $\alpha$-stable case from Example \ref{Ex:alphastablecase} bounded convergence yields
\begin{multline}\label{E:FKIalphastable}
\left\langle P_t^{a,v}f,g\right\rangle_{L^2(X,\mu)}=\lim_{\varepsilon \to 0}\mathbb{E}_{\overline{g}dx}\left[e^{-V_t}\cos \left( \sum_{0<s\leq t}a(Y_s,Y_{s-})\mathbf{1}_{\left\lbrace |a(Y_s,Y_{s-}|>\varepsilon\right\rbrace}+I_{t,\varepsilon}\right)\right.\notag\\
+\left. ie^{-V_t} \sin \left( \sum_{0<s\leq t}a(Y_s,Y_{s-})\mathbf{1}_{\left\lbrace |a(Y_s,Y_{s-}|>\varepsilon\right\rbrace}+I_{t,\varepsilon}\right)\right]
\end{multline}
for any $f\in L^2(\mathbb{R}^n)$ and $g\in L^2(\mathbb{R}^n)\cap L^\infty(\mathbb{R}^n)$ with $I_{t,\varepsilon}=\int_0^t\int_{\left\lbrace y: |a(y,Y_s)|\leq \varepsilon\right\rbrace} \frac{a(y,Y_s)}{|y-Y_s|^{n+\alpha}}dyds$ and $V_t:=\int_0^t v(Y_s)ds$.
\end{examples}

In general the semigroup $(P_t)_{t>0}$ will not be positivity preserving and in particular not Markovian. However, the following \emph{diamagnetic inequalities} are immediate from (\ref{E:FKI}).  

\begin{corollary}\label{C:diamag}
Let $a\in\mathcal{H}_a$ be real and let $v$ be a real valued Borel function such that $v_-$ is uniformly bounded. Then we have
\[|P_t^{a,v}f(x)|\leq P_t^{0,v}|f|(x)\leq e^{t\left\|v_-\right\|_{\sup}}P_t|f|(x), \ \ t>0,\ \ x\in X,\]
for any bounded Borel function $f$.
\end{corollary}

By standard theory there are a unique self-adjoint operator $(H^{a,v}, \dom (H^{a,v}))$ and a unique closed conjugate symmetric bilinear form $(\mathcal{Q}^{a,v},\dom \mathcal{Q}^{a,v})$ on $L^2(X,m)$
such that 
\[H^{a,v}f=- \lim_{t\to 0}\frac{1}{t} (f-P^{a,v}_tf)\ \ \text{ and }\ \ \mathcal{Q}^{a,v}(f,g)=\left\langle H^{a,v}f,g\right\rangle_{L^2(X,\mu)}\]
for any $f\in \dom H^{a,v}$ and $v\in \dom \mathcal{Q}^{a,v}$.

\section{Identification and closability}\label{S:Ident}

Under additional conditions the form $\mathcal{Q}^{a,v}$ appears as a closed extension of  $\mathcal{E}^{a,v}$ as in (\ref{E:electric}). We say that $(\mathcal{E},\mathcal{F})$ admits a \emph{carr\'e du champ} if all energy measures $\Gamma(f)$, $f\in\mathcal{C}$, are absolutely continuous with respect to $\mu$. In this case they have $\mu$-integrable densities $x\mapsto \Gamma(f)(x)$. See \cite[Chapter I]{BH91}. Set $\mathcal{D}_L:=\left\lbrace f\in \dom L\cap L^1(X,\mu)\cap L^\infty(X,\mu): Lf\in L^1(X,\mu)\right\rbrace$.

\begin{theorem}\label{T:coincide}
Let $(\mathcal{E},\mathcal{F})$ be a conservative regular Dirichlet form on $L^2(X,\mu)$ with generator $(L,\dom L)$. Assume that it admits a carr\'e du champ and that its jump measure $J$ is of form $J(d(x,y))=\frac12n(x,dy)\mu(dx)$
with a kernel $n(x,dy)$ on $(X,\mathcal{B}(X))$. Let $a\in \mathcal{H}_a$ be real and of form $a=\partial u+\eta$ with $u\in \dom L$ and $\eta \in \ker \partial^\ast$. Let $v$ be a real valued Borel function with uniformly bounded negative part $v_-$. Then 
\begin{multline}\label{E:firstrepHav}
\left\langle H^{a,v}f,g\right\rangle=\left\langle (\partial_c+ia_c)f,(\partial_c+ia_c)g\right\rangle_{\mathcal{H}_c}\\
+\int_X\int_X(f(x)-e^{ia_j(x,y)}f(y))n(x,dy)\overline{g(x)}\mu(dx)+\left\langle vf,g\right\rangle
\end{multline}
for all $f\in \mathcal{D}_L$ and $g\in\mathcal{F}\cap L^\infty(X,\mu)$. 
\end{theorem}

The collection of vector fields $a=\partial u+\eta$ with $u\in \dom L$ and $\eta\in \mathcal{H}_a\cap \ker \partial^\ast$ as considered in Theorem \ref{T:coincide} is dense in $\mathcal{H}_a$. This follows from (\ref{E:Hodge}) and from the density of $\dom L$ in $\mathcal{F}$. Recall that the semigroup $(P_t)_{t\geq 0}$ is called a \emph{Feller semigroup} if it is a strongly continuous semigroup of (in this case) contractions on the space $C_0(X)$ of continuous functions vanishing at infinity. If it is Feller, we denote its $C_0(X)$-generator by $(L, \dom_{C_0(X)} L)$. We assume that $\dom_{C_0(X)}$ contains sufficiently many compactly supported functions.

\begin{theorem}\label{T:coincide2}
Let the hypotheses of Theorem \ref{T:coincide} be in force. In addition assume that $(P_t)_{t>0}$ is Feller and that $\mathcal{C}_L:=\dom_{C_0(X)} L\cap \mathcal{C}$
is dense in $L^2(X,\mu)$. Then $\mathcal{C}_L\subset \mathcal{D}_L$,
\[\mathcal{Q}^{a,v}(f,g)=\mathcal{E}^{a,v}(f,g),\ \ f,g\in\mathcal{C}_L,\]
and $(\mathcal{Q}^{a,v},\dom \mathcal{Q}^{a,v})$ is a closed extension of $(\mathcal{E}^{a,v},\mathcal{C}_L)$. For the associated non-negative self-adjoint operator $(H^{a,v}, \dom (H^{a,v}))$ we have $\mathcal{C}_L\subset \dom H^ {a,v}$ and
\begin{equation}\label{E:Havexplicit}
H^{a,v}f(x)=(\partial_c+ia_c)^\ast (\partial_c+ia_c)f(x)+\int_X(f(x)-e^{ia_j(x,y)}f(y))n(x,dy)+v(x)f(x),
\end{equation}
$f\in \mathcal{C}_L$. 
\end{theorem}

\begin{examples}
For the isotropic $\alpha$-stable case from Example \ref{Ex:alphastablecase} we have $C_c^2(\mathbb{R}^n)\subset \dom H^{a,v}$ and
\[H^{a,v}f(x)=\int_X\frac{f(x)-e^{ia(x,y)}f(y)}{|x-y|^{n+\alpha}}dy + v(x)f(x), \ \ f\in \dom H^{a,v}. \]
\end{examples}

There is also an approximation for $H^{a,v}$ in terms of the semigroup $(P_t)_{t\geq 0}$, similar to (\ref{E:genformula}).

\begin{corollary}\label{C:sgapproxmagnetic}
Under the hypotheses of Theorem \ref{T:coincide2} we have
\[H^{a,v}f(x)=\lim_{t\to 0}\frac{1}{t}\int_X(f(x)-e^{ia(x,y)}f(y))P_t(x,dy)+v(x)f(x),\ \ f\in \mathcal{C}_L.\]
\end{corollary}

Corollary \ref{C:diamag} implies an estimate that generalizes an inequality in \cite{Si76}, see also \cite[p. 2]{Simon79}.

\begin{corollary}\label{C:diamagenergy}
Under the conditions of Theorem \ref{T:coincide2} we have $\mathcal{E}^{0,v}(|f|)\leq \mathcal{E}^{a,v}(f)$,  $f\in\mathcal{C}_L$, and in particular, $\mathcal{E}(|f|)\leq \mathcal{E}^{a,0}(f)$.
\end{corollary}
\begin{proof}
For $f\in\mathcal{C}_L$ we have $|f|\in\mathcal{C}$ and $|\Re \left\langle P_t^{a,v}f,f\right\rangle_{L^2(X,\mu)}|\leq \left\langle P_t^{0,v}|f|,|f|\right\rangle_{L^2(X,\mu)}$.
Therefore $\left\langle f-P_t^{a,v}f,f\right\rangle_{L^2(X,\mu)}\geq \left\langle |f|-P_t^{0,v}|f|,|f|\right\rangle_{L^2(X,\mu)}$, what implies the statement by semigroup approximation. 
\end{proof}

We prove Theorems \ref{T:coincide} and \ref{T:coincide2} and Corollary \ref{C:sgapproxmagnetic}. Given $\omega\in\mathcal{H}$ let $M=\Theta(\omega)$. By Theorem \ref{T:Nakao} we have $\mu_{\left\langle M\right\rangle}(dx)=2\Gamma_{\mathcal{H}}(\omega)(dx)$,
where $\Gamma_\mathcal{H}(\omega)$ is the energy measure of $a$ as in (\ref{E:GammaH}).
Since $(\mathcal{E},\mathcal{F})$ admits a carr\'e du champ, also the energy measure $\Gamma_\mathcal{H}(\omega)$ has a $\mu$-integrable density $x\mapsto \Gamma_{\mathcal{H}}(\omega)(x)$, see for instance \cite[Section 2]{HRT}, hence $\mu_{\left\langle M\right\rangle}(dx)=2\Gamma_{\mathcal{H}}(\omega)(x)\mu(dx)$. 
By $x\mapsto \Gamma_{\mathcal{H},c}(\omega)(x)$ we denote the density of the
strongly local part $\Gamma_{\mathcal{H},c}(\omega)$ of $\Gamma_{\mathcal{H}}(\omega)$ as defined in (\ref{E:GammaHc}). As $(P_t)_{t>0}$ is Markovian and $\mu$-symmetric, we have 
\[\int_X h(x)\overline{P_s g(x)}\mu(dx)=\int_X P_s h(x)\overline{g(x)}\mu(dx)\]
for any $s>0$, any any $h\in L^\infty(X,\mu)$ and $g\in L^1(X,\mu)$. Therefore 
\begin{align}
\int_X \mathbb{E}_x(\left\langle M\right\rangle_t)\overline{g(x)}\mu(dx)&=2\int_0^t\int_X  \Gamma_{\mathcal{H}}(\omega)(x) \overline{P_s g(x)}\mu(dx)ds\notag\\
&=2\int_0^t\int_X P_s\Gamma_{\mathcal{H}}(\omega)(x) \overline{g(x)}\mu(dx)ds\notag\\
&=2\int_X\left(\mathbb{E}_x\int_0^t \Gamma_{\mathcal{H}}(\omega)(Y_s)ds \right)\overline{g(x)}\mu(dx)\notag
\end{align}
for any $h\in L^\infty(X,\mu)$. Clearly $t\mapsto \int_0^t \Gamma_{\mathcal{H}}(\omega)(Y_s)ds$ is a positive continuous AF. Therefore the uniqueness in the Revuz correspondence and polarization yield the following.

\begin{lemma}\label{L:brackets}
Given $\omega_1,\omega_2\in\mathcal{H}$ set $M_1:=\Theta(\omega_1)$ and $M_2:=\Theta(\omega_2)$. Then 
\[\left\langle M_1, M_2\right\rangle_t=2\int_0^ t\Gamma_{\mathcal{H}}(\omega_1,\omega_2)(Y_s)ds\ \ \text{ and }\ \ 
\left\langle M_1^c, M_2^c\right\rangle_t=2\int_0^ t\Gamma_{\mathcal{H},c}(\omega_1,\omega_2)(Y_s)ds.\]
\end{lemma}

Standard Girsanov type arguments yield a first explicit representation for $H^{a,v}$. Recall that $L_c$, $\partial_c$, $\partial_c^\ast$ and $\Gamma_{\mathcal{H},c}$ denote the strongly local parts of $L$, $\partial$, $\partial^\ast$ and $\Gamma_{\mathcal{H}}$. 

\begin{lemma}
Let the hypotheses of Theorem \ref{T:coincide} be in force. Then
\begin{multline}\label{E:firstexplicit}
H^{a,v}f(x)=-L_c f(x)-2i\Gamma_{\mathcal{H},c}(\partial_c f, a_c)(x)+i\partial_c^\ast a(x)f(x)+\Gamma_{\mathcal{H},c}(a_c)(x)f(x)\\
+\int_X (f(x)-e^{ia_j(x,y)}f(y))n(x,dy) +v(x)f(x)
\end{multline}
for all $f\in \mathcal{D}_L$, seen as an equality in $L^1(X,\mu)$.
\end{lemma}

\begin{examples}
In the classical case of the $n$-dimensional Brownian motion from Examples \ref{Ex:sum} (i) we observe 
\[H^{a,v}f=-\frac12\Delta f-ia\cdot\nabla f-\frac{i}{2}(\diverg a)f+\frac12 a^2f+vf.\]
\end{examples}

\begin{proof}
Let $M=\Theta(a)$. We write $S_t=\int_{Y([0,t])}a$ and $V_t:=\int_0^t v(Y_s)ds$ and set
$C_t:=S_t-V_t$. By Remark \ref{R:cosycase2} (iii) the process $C=(C_t)_{t\geq 0}$ is a cadlag $\mathbb{P}_x$-semimartingale for q.e. $x\in X$, and by It\^o's formula, \cite[Section II.7, Theorem 32]{Protter}, 
\[e^{iC_t}=i\int_0^te^{iC_{s-}}dC_s-\frac12\int_0^te^{iC_{s-}}d\left\langle M^c\right\rangle_s
+\sum_{0<s\leq t}\left\lbrace e^{iC_s}-e^{iC_{s-}}-ie^{C_{s-}}(C_s-C_{s-})\right\rbrace +1\]
$\mathbb{P}_x$-a.s. for q.e. $x\in X$.
The mutual variation of $(f(Y_t))_{t\geq 0}$ and $(e^{iC_t})_{t\geq 0}$ is given by
\[\left[f(Y),e^{iC}\right]_t=\left\langle f(Y)^c, (e^{iC})^c\right\rangle_t+f(Y_0)+\sum_{0<s\leq t}(f(Y_s)-f(Y_{s-}))(e^{iC_s}-e^{iC_{s-}}),\]
see \cite[Section II.6]{Protter}. Integrating by parts, \cite[Section II.6, Corollary 2]{Protter},
\begin{align}
e^{iC_s}f(Y_t)=f(Y_0)&+i\int_0^t f(Y_{s-})e^{iC_{s-}}dC_s-\frac12\int_0^tf(Y_{s-})e^{iC_{s-}}d\left\langle M^c\right\rangle_s\notag\\
&+\sum_{0<s\leq t}f(Y_{s-})\left\lbrace e^{iC_s}-e^{iC_{s-}}-ie^{C_{s-}}(C_s-C_{s-})\right\rbrace\notag\\
&+\int_0^t e^{iC_{s-}}dM_s^f +\int_0^t e^{iC_{s-}}dN_s^f\notag\\
&+ i\int_0^t e^{iC_{s-}} d\left\langle M^{f,c}, M^c\right\rangle_s\notag\\
&+ \sum_{0<s\leq t}(f(Y_s)-f(Y_{s-}))(e^{iC_s}-e^{iC_{s-}}),
\end{align}
$\mathbb{P}_x$-a.s for q.e. $x\in X$. 
We have $\partial_j^\ast a_j(x)=\int a_j(x,y)n(x,dy)$
so that
\[\int_0^t f(Y_{s-})e^{iC_{s-}}(\partial_j^\ast a)(Y_s)ds=\int_0^tf(Y_{s-})e^{iC_{s-}}\int_X a_j(Y_s,y)n(Y_s,dy)ds\]
and 
\begin{align}
i\int_0^t f(Y_{s-})e^{iC_{s-}}dC_s&=i\int_0^ t f(Y_{s-})e^{iC_{s-}}dM_s-i\int_0^tf(Y_{s-})e^{iC_{s-}}(\partial_c^\ast a)(Y_s)ds\notag\\
&-i\int_0^t f(Y_{s-})e^{iC_{s-}}\int_X a_j(Y_s,y)n(Y_s,dy)ds.\notag
\end{align}
Now recall (\ref{E:Lj}). Taking into account Lemma \ref{L:brackets}, $e^{iC_t}f(Y_t)$ is seen to equal
\begin{align}\label{E:collect}
f(Y_0)&+i\int_0^tf(Y_{s-})e^{iC_{s-}}dM_s-i\int_0^tf(Y_{s-})e^{iC_{s-}}(\partial^\ast_c a)(Y_s)ds
-\int_0^t f(Y_{s-})v(Y_s)ds\notag\\
&-\int_0^t f(Y_{s-})e^{iC_{s-}}\Gamma_{\mathcal{H}_c}(a_c)(Y_s)ds+\int_0^t e^{iC_{s-}}dM_s^{f}+\int_0^t e^{iC_{s-}}L^cf(Y_s)ds\notag\\
&+2i\int_0^t  e^{iC_{s-}}\Gamma_{\mathcal{H},c}(\partial_c f, a_c)(Y_s)ds\notag\\
&-i\int_0^t f(Y_{s-})e^{iC_{s-}}\int_X a_j(Y_s,y)n(Y_s,dy)ds+\int_0^t e^{iC_{s-}}\int_X (f(y)-f(Y_s))n(Y_s,dy)\notag\\
&-i\sum_{0<s\leq t}f(Y_{s-})e^{iC_{s-}}(C_s-C_{s-}) + \sum_{0<s\leq t}f(Y_s)(e^{iC_s}-e^{iC_{s-}})\notag
\end{align}
$\mathbb{P}_x$-a.s. for q.e. $x\in X$. Taking expectations with respect to $\mathbb{P}_x$ the martingale terms vanish. Moreover, the processes $t\mapsto f(Y_{t-})e^{iC_{t-}}$ and $t\mapsto e^{iC_{t-}}$ are left continuous, hence they are predictable, and using the L\'evy system formula (\ref{E:LSF}) we therefore obtain
\begin{multline}
\mathbb{E}_x[e^{iC_t}f(Y_t)]=f(x)+\int_0^t\mathbb{E}_x\left[e^{iC_{s-}}\left\lbrace L^cf(Y_s)+2i\Gamma_{\mathcal{H},c}(\partial f, a)(Y_s)-i\partial_c^\ast a(Y_s)f(Y_{s-})\right.\right.\notag\\
\left.\left.- \Gamma_{\mathcal{H},c}(a)(Y_s)f(Y_{s-})-v(Y_s)f(Y_{s-})\right\rbrace\right]ds\notag\\
+\mathbb{E}_x\left[\sum_{0<s\leq t}\left\lbrace f(Y_s)(e^{iC_s}-e^{iC_{s-}})+e^{iC_{s-}}(f(Y_s)-f(Y_{s-}))\right\rbrace\right],\notag
\end{multline}
note that $C_s-C_{s-}=M_s-M_{s-}=a_j(Y_{s-},Y_s)$. The last summand rewrites
\[\mathbb{E}_x\left[\sum_{0<s\leq t}(e^{iC_s}f(Y_s)-e^{iC_{s-}}f(Y_{s-}))\right]=\mathbb{E}_x\left[\sum_{0<s\leq t}e^{iC_{s-}}(e^{ia(Y_{s-},Y_s)}f(Y_s)-f(Y_{s-}))\right],\]
what by (\ref{E:LSF}) equals $\mathbb{E}_x\left[\int_0^t e^{iC_{s-}}\int_X (e^{ia_j(Y_s,y)}f(y)-f(Y_s))n(Y_s,dy)ds\right]$. By Lemma \ref{L:replace} integration against $\overline{g}d\mu$ with $g\in L^\infty(X,\mu)$ shows that $\left\langle P_t^{a,v}f,g\right\rangle$ equals
\begin{multline}
\left\langle f, g\right\rangle +\int_0^t \left\langle P_s^{a,v}\left\lbrace L^cf+2i\Gamma_{\mathcal{H},c}(\partial f,a)-i(\partial_c^\ast a)f-\Gamma_{\mathcal{H},c}(a)f-vf\right.\right.\notag\\
\left.\left.+ \int_X (e^{ia_j(\cdot, y)}f(y)-f(\cdot))n(\cdot,dy)\right\rbrace, g\right\rangle ds\notag
\end{multline}
By Theorem \ref{T:semigroup} the integrand is continuous in $s$, hence $\left\langle H^{a,v}f,g \right\rangle$ is the limit as $t$ goes to zero of $t^{-1}$ times the integral on the right hand side. Using Hahn-Banach we may conclude formula (\ref{E:firstexplicit}).
\end{proof}

Theorem \ref{T:coincide} now follows easily: By the Leibniz rule (\ref{E:Leibniz}), \[\left\langle (\partial_c^\ast a)f,g\right\rangle=\left\langle f, (\partial^\ast_c a)g\right\rangle_{L^2(X,\mu)}=\left\langle \partial (\overline{g}f),a\right\rangle_{\mathcal{H}_c}=\left\langle \partial f, ag\right\rangle_{\mathcal{H}_c}+\left\langle af,\partial g\right\rangle_{\mathcal{H}_c}.\] With $\left\langle \Gamma_{\mathcal{H},c}(\partial f, a),g\right\rangle=\left\langle \partial f, ag\right\rangle_{\mathcal{H}_c}$
and
$\left\langle \Gamma_{\mathcal{H},c}(a)f,g\right\rangle=\left\langle af, ag\right\rangle_{\mathcal{H}_c}$,
we obtain
\begin{multline}
\left\langle H^{a,v}f,g\right\rangle=\left\langle \partial_c f,\partial_c g\right\rangle_{\mathcal{H}_c}+i\left\langle\partial_c f, ag\right\rangle_{\mathcal{H}_c}+i\left\langle af, \partial g\right\rangle_{\mathcal{H}_c}+\left\langle af,ag\right\rangle_{\mathcal{H}_c} +\left\langle vf,g\right\rangle\notag\\
+\int_X\int_X (f(x)-e^{ia_j(x,y)}f(y))n(x,dy)\overline{g(x)}\mu(dx).\notag
\end{multline}

We provide the arguments for Theorem \ref{T:coincide2}. For $f\in\mathcal{C}_L$ we have $H^{a,v}f\in L^2(X,\mu)$, hence $\mathcal{Q}^{a,v}(f,g)=\left\langle H^{a,v}f,g\right\rangle_{L^2(X,\mu)}=\left\langle H^{a,v}f,g\right\rangle$, $g\in \mathcal{C}_L$. On the other hand we observe 
$\left\langle H^{a,v}f,g\right\rangle=\mathcal{E}^{a,v}(f,g)$
since
\begin{align}\label{E:shiftover}
-\int_X\int_X(f(x)-e^{ia_j(x,y)}f(y))&\overline{e^{ia_j(x,y)}g(y)}n(x,dy)\mu(dx)\\
&=\int_X\int_X(f(y)-e^{-ia_j(x,y)}f(x))\overline{g(y)}n(x,dy)\mu(dx)\notag\\
&=\int_X\int_X(f(x)-e^{ia_j(x,y)}f(y))\overline{g(x)}n(x,dy)\mu(dx),\notag
\end{align}
where have used the antisymmetry of $a_j$ and the symmetry of $n(x,dy)\mu(dx)$ in the second equality. The arguments of \cite[Proposition 4.1]{HTb} yield a first interpretation of (\ref{E:Havexplicit}) in the topological dual of the space $\mathcal{C}_L$ (endowed with the norm $f\mapsto \mathcal{E}(f)^{1/2}+\left\|f\right\|_{L^2(X,\mu)}+\left\|f\right\|_{\sup}$). Combined with the fact that $H^{a,v}f\in L^2(X,\mu)$ we arrive at (\ref{E:Havexplicit}).

Corollary \ref{C:sgapproxmagnetic} follows by considering (\ref{E:shiftover}) with $P_t(x,dy)$ in place of $n(x,dy)$.

\end{document}